\documentclass[12pt,reqno]{amsart}

\usepackage[cp1251]{inputenc}
\usepackage[english]{babel}
\usepackage{amsmath,amsxtra,amssymb,eufrak}
\usepackage[all]{xy}
\usepackage{mathrsfs}
\usepackage{paralist}
\usepackage[hypertex]{hyperref}
\usepackage[active]{srcltx} 

\newtheorem{thm}{Theorem}[section]
\newtheorem{lm}[thm]{Lemma}
\newtheorem{co}[thm]{Corollary}
\newtheorem{pr}[thm]{Proposition}

\theoremstyle{definition}

\newtheorem{df}[thm]{Definition}
\newtheorem{exm}[thm]{Example}
\newtheorem{exms}[thm]{Examples}
\newtheorem{rem}[thm]{Remark}
\newtheorem{nota}[thm]{Notation}

\numberwithin{equation}{section}

\DeclareMathOperator{\diag}{diag}
\DeclareMathOperator{\Ker}{Ker}
\DeclareMathOperator{\Rad}{Rad}
\DeclareMathOperator{\Sp}{Sp}
\DeclareMathOperator{\ad}{ad}

\newcommand{\id}{\mathop{\mathrm{id}}\nolimits}
\newcommand{\ptn}{\mathbin{\widehat{\otimes}}}

\title{Envelopes in the class of Banach algebras of polynomial growth and $C^\infty$-functions of a finite number of free variables}

\author{O. Yu. Aristov}
\address{Institute for Advanced Study in Mathematics of Harbin Institute of Technology, Harbin 150001, China;
\newline\indent
Suzhou Research Institute of Harbin Institute of Technology, Suzhou 215104, China}
\email{aristovoyu@inbox.ru}

\begin{document}

\begin{abstract}
We introduce the notion of envelope of a topological algebra (in particular, an arbitrary associative algebra) with respect to a class of Banach algebras. In the case of the class of real Banach algebras of polynomial growth, i.e., admitting a $C^\infty$-functional calculus for every element, we get a functor that maps the algebra of polynomials in $k$ variables to the algebra of $C^\infty$-functions on $\mathbb{R}^k$. The envelope of a general commutative or non-commutative algebra can be treated as an algebra of $C^\infty$-functions on some commutative or non-commutative space. In particular, we describe the envelopes of the universal enveloping algebra of finite-dimensional Lie algebras, the coordinate algebras of the quantum plane and quantum group $SL(2)$ and also look at some commutative examples. A result on algebras of `free $C^\infty$-functions', i.e., the envelopes of  free associative algebras of finite rank $k$, is announced for general $k$ and proved for $k\leqslant 2$.
\end{abstract}

\maketitle

\markright{Envelopes and $C^\infty$-functions of free variables}

{\small\it
\hfill Dedicated to my beloved teacher,}

{\small\it
\hfill Professor Alexander Helemskii,}

{\small\it
\hfill  on the occasion of his 80th birthday}
\bigskip

The Arens--Michael envelope can be treated as a functor from the algebraic geometry to the analytic geometry (at the level of function algebras). In particular, it corresponds to the analytization of an affine scheme of finite type; see \cite{Pi15}. An advantage of this view on the analytization is that the functor is well defined even for non-commutative algebras. Although this approach also has its drawbacks (for example,  algebras with rich representation theory such as the Weyl algebra can vanish), it works well in many cases; see, e.g., \cite{ArAMN}. We are interested in finding a similar functor from the algebraic geometry to the differential geometry, that is, from the category of associative real algebras to some category of real topological associative  algebras that exhibits the same behaviour as the Arens--Michael envelope in the analytic case. This means that it must transform the algebra of real polynomials $\mathbb{R}[x_1,\ldots ,x_k]$, $k\in\mathbb{N}$, into the algebra $C^\infty(\mathbb{R}^k)$ of real-valued infinitely differentiable functions. Moreover, it is natural to expect that the functor should act non-trivially on a number of classical examples of non-commutative algebras.

Here we consider a construction which, at first sight, should be even more limited in its possibilities  than the Arens--Michael envelope  but, in fact, its application to certain classical non-commutative algebras demonstrates the non-triviality of the theory and links it to various areas of algebra and analysis. Namely, we propose to use the envelope with respect to the class of Banach algebras of polynomial growth as a bridge to the $C^\infty$-case. A real unital Banach algebra $B$ is said to be of \emph{polynomial growth} if for every $b\in B$ there are $K>0$ and $\alpha\geqslant0$ such that $\|e^{isb}\|\leqslant K (1+|s|)^{\alpha}$ for all $s\in\mathbb{R}$; equivalently,  $B$ admits a $C^\infty$-functional calculus associated with $b$; see~\cite{ArOld}. (Note that the theory only makes sense if the ground field is~$\mathbb{R}$).

The axiomatic approach based on the class of Banach algebras of polynomial growth is indeed adequate to the problem since the natural embedding of $\mathbb{R}[x_1,\ldots,x_m]\to C^\infty(\mathbb{R}^m)$ turns out to be an envelope with respect to this class; see Proposition~\ref{envpol}. By definition, the envelope is an algebra that is locally of polynomial growth (contained in the class $\mathsf{PGL}$ of projective limits of algebras having polynomial growth) and is universal in the sense that any homomorphism into a Banach algebra of polynomial growth factors through it. The envelope exists for each  real associative algebra $A$ and, furthermore, for each real topological algebra.  In fact, first examples were considered by the author in~\cite{ArOld}, although the corresponding terminology was not introduced there.

Note that  Banach algebras of polynomial growth form a rather narrow class since each of them is commutative modulo Jacobson radical and, moreover, the radical is nilpotent; see~\cite{ArOld} and also Theorem~\ref{RadPG} below. In particular, all irreducible representations are one-dimensional. Despite this fact, the author believes that this article shows that the theory of envelopes turns out to be quite rich.  Constructing an algebra of free non-commutative functions of class $C^\infty$ as an envelope of an algebra of free polynomials (i.e., a free associative algebra) is of particular interest. Our central result, Theorem~\ref{fuctcafree}, gives an explicit description of this algebra of `free $C^\infty$-functions' of finite rank~$k$. We denote it by $C^\infty_{\mathfrak{f}_k}$. As a Fr\'echet space, $C^\infty_{\mathfrak{f}_k}$ has the form
$$
C^\infty(\mathbb{R}^k)\ptn \Bigl(\,\prod_{n=0}^\infty C^\infty(X)^{\ptn n}\Bigr),
$$
where $\ptn$ denotes the projective tensor product and $X$ is a disjoint union of vector spaces.  The proof of this theorem is quite technical and surprisingly non-trivial. In this paper, we only prove the theorem only for $k=2$ (when $k=1$ it is trivial modulo known results). A proof for the general case is given in~\cite{ArNew2}.

Being a projective limit of Banach algebras of polynomial growth, $C^\infty_{\mathfrak{f}_k}$ also has only one-dimensional irreducible representations and its Gelfand spectrum can be identified with~$\mathbb{R}^k$. This algebra looks rather modest compared to the algebra of `free entire functions' \cite{Ta72,Ta73}, which has not only finite-dimen\-sional irreducible representations of arbitrary dimensions, but also infinite-dimensional representations in abundance. However, $C^\infty_{\mathfrak{f}_k}$ has a plenty of triangular representations, which compensates for the above disadvantage to some extent.

Note that the algebras of ‘smooth functions’ traditionally considered in the non-commu\-ta\-tive differential geometry in the spirit of Connes are usually self-adjoint subalgebras of $C^*$-algebras (in particular, their radicals are trivial). But, in our reasoning, algebras of triangular matrices are fundamental and, as a consequence, the examples discussed in this paper admit a (trivial) involution only in the commutative case. This restriction seems  to be an essential feature of our approach. Furthermore, the presented theory has the shortcoming that it does not include some classical examples, such as the non-commutative tori (since complex numbers are used in their definition) and the quantum disc (see Example~\ref{qudi}).

It is worth mentioning  other axiomatic approaches to non-commutative algebras with some $C^\infty$ structure. The first theory uses differential sequences of seminorms; see the version of Blackadar and Cuntz  in \cite{BC91} and the version of Kissin and Shulman in \cite{KSh}. Unfortunately, within this framework, it has not yet been possible either to prove the existence of a $C^\infty$-functional calculus or to construct a counterexample.  There is a functional calculus for a somewhat narrower class~\cite{KSh}, but we cannot yet say more. The second approach, which uses systems of partial derivatives, is proposed by Akbarov \cite{Ak15,Ak17}. Despite the fact that some interesting results have been obtained, this theory does not seem promising for non-commutative algebras.

Although theoretical constructions occupy a considerable part of this text, the emphasis is on examples. In addition to free algebras, we consider the universal enveloping algebras of finite-dimensional Lie algebras (refining results in~\cite{ArOld}), the coordinate algebras of the quantum plane and quantum group $SL_q(2,\mathbb{R})$; see \S\,\ref{s:EnvComm}. The commutative case is also of interest and discussed there. We examine the general properties of the class $\mathsf{PGL}$ in \S\,\ref{s:BanPG} and also give a universal construction of $C^\infty$-tensor algebra in \S\,1. Note that the $C^\infty$-tensor algebra of a $k$-dimensional space is exactly the algebra of `free $C^\infty$-functions' of rank~$k$.

As regards abstract considerations, a general theory of envelopes with respect to a class of Banach algebras is developed in \S\,\ref{s:envel}.  Until very recently, only the envelopes with respect to the class of all Banach algebras, the Arens--Michael envelopes, were considered. However, to develop a satisfactory theory of envelopes it suffices to take an arbitrary class of Banach algebras that is stable under passing to finite products and closed subalgebras. Also, it is desirable that this class to be stable under passing to quotients. This assumption has a certain geometric sense but we will not discuss this topic here. The following classes satisfy all three conditions:

--- the class $\mathsf{PG}$ of Banach algebras of polynomial growth;

--- the class of strictly real Banach algebras in the sense of~\cite{In64};

--- the class of Banach algebras satisfying a polynomial identity.

Note that (in the case of real algebras) the second and third classes both contain $\mathsf{PG}$, which plays major  role in this article. The enveloping functor with respect to the third class is introduced in~\cite{ArDAAF} and also some examples are given there. The stability properties for strictly real Banach algebras are proved in \cite[Theorem 4.4]{Mi93} but the envelopes have not been studied yet. Note also that for algebras with involution, envelopes  with respect to the class of $C^*$-algebras have been studied earlier; see, for example,  \cite{Fr81,Ku13}.

The author wishes to thank the Institute of Mathematics and Mechanics at the Kazan Federal University and especially Renat Gumerov for the hospitality during the visit to Kazan in 2023.

\tableofcontents

\section{Local containment and envelopes}
\label{s:envel}

We denote by $\mathsf{TA}$ the category of topological algebras over $\mathbb{C}$ or $\mathbb{R}$, that is, topological vector spaces endowed with separately continuous multiplication; see \cite[Chapter~1, \S\,1, p.\,4, Definition 1.1]{Ma86} or \cite[\S\,1, p.\,6, Definition 1.6]{Fr05}.  In general, the underlying vector space is not assumed  to be locally convex or complete and the existence of identity is also not assumed. Here the morphisms are continuous homomorphisms. Note that a subclass of objects in $\mathsf{TA}$ can be treated as a full subcategory. Following the tradition of functional analysis, we use the somewhat old-fashioned term \emph{projective limit} (also known as `inverse limit') for what is called \emph{directed limit} in category theory.

\begin{df}\label{loccond}
Let  $\mathsf{C}$ be a class of Banach algebras. We say that a topological algebra is a $\mathsf{CL}$ algebra or \emph{locally in $\mathsf {C}$} if it is isomorphic to a projective limit (in $\mathsf{TA}$) of algebras contained in $\mathsf{C}$.
\end{df}

Recall that a complete topological algebra whose topology  can be determined by
a family of submultiplicative seminorms is called an \emph{Arens--Michael algebra}. If $\mathsf{C}$ is contained in the class of all Banach algebras, then each algebra locally contained in $\mathsf{C}$ is obviously an Arens--Michael algebra. Furthermore, an Arens--Michael algebra is exactly an algebra that is locally in the class of all Banach algebras. Note also that the limit of a projective system of Banach algebras in $\mathsf{TA}$ coincides with the limit in the category of Arens--Michael algebras; see, e.g., \cite[Chapter~3, \S\,2, p.\,84, (2.8)]{Ma86}.

When considering envelopes, we always assume by default that an algebra without an initially given topology is endowed with the strongest locally convex one. Thus every associative algebra can be treated as a topological algebra.

\begin{df}\label{defengen}
Let $\mathsf{C}$ be a class of Banach algebras  (over $\mathbb{C}$ or $\mathbb{R}$), $\mathsf{CL}$ be defined as in Definition~\ref{loccond}, $F\!:\mathsf{CL}\to \mathsf{TA}$ be a corresponding forgetful functor. If $F$ admits a left adjoint functor~$L$, then the \emph{envelope with respect to the class $\mathsf{C}$} of a topological algebra~$A$ is the pair $(\widehat A^{\,\mathsf{C}}, \iota_A)$, where $\iota_A$ is the component of the identity adjunction (the functor morphism $\iota\!:\id_{\mathsf{TA}}\Rightarrow F\circ L$) corresponding to $A$, and $\widehat A^{\,\mathsf{C}}\!:=FL(A)$.
\end{df}

Of course, the definition makes sense not only for $\mathsf{CL}$ but also for an arbitrary full subcategory in $\mathsf{TA}$, but we do not need it here.

The choice of the term `envelope' is more likely due to historical considerations; cf. \cite[Chapter~5]{X2}. It could just as well be called  'free functor'. However, we reserve the last term for those cases where the forgetful functor changes the algebraic type of the category (such as `associative algebras $\rightarrow$ vector spaces' or `associative algebras $\rightarrow$ sets'). Akbarov in his papers \cite{Ak15,Ak17} also uses the term `envelope', which has a similar but different meaning.

Note also that the homomorphism  $A\to\widehat A^{\,\mathsf{C}}$ does not have to be injective; see Example~\ref{nontri}(A).

\begin{rem}\label{simpdef}
Definition~\ref{defengen} can be formulated without the use of adjoint functors as follows.
An \emph{envelope with respect to the class $\mathsf{C}$} is a pair $(\widehat A^{\,\mathsf{C}}, \iota_A)$, where $\widehat A^{\,\mathsf{C}}$ is locally in  $\mathsf{C}$ and $\iota_A$ is a continuous homomorphism $A \to \widehat A^{\,\mathsf{C}}$ if for every~$B$ in $\mathsf{C}$ and every continuous homomorphism $\varphi\!: A \to B$ there is a unique continuous homomorphism
$\widehat\varphi\!:\widehat A^{\,\mathsf{C}} \to B$ such that the diagram
\begin{equation*}
  \xymatrix{
A \ar[r]^{\iota_A}\ar[rd]_{\varphi}&\widehat A^{\,\mathsf{C}}\ar@{-->}[d]^{\widehat\varphi}\\
 &B\\
 }
\end{equation*}
is commutative.

Note that according to Definition~\ref{defengen} we take $B$ in $\mathsf{CL}$ but it suffices to take it in $\mathsf {C}$.
\end{rem}

The following proposition gives sufficient conditions for the existence of an envelope.

\begin{pr}\label{exisenvgen}
Let $\mathsf{C}$ be a class of Banach algebras stable under passing to finite products and closed subalgebras. Then the envelope  with respect to $\mathsf{C}$ exists and is unique up to natural isomorphism for every topological algebra.
\end{pr}
\begin{proof}
Given a topological algebra $A$, consider the set of all continuous sub\-multi\-pli\-cative seminorms on~$A$, the completions relative to which belong to $\mathsf{C}$ and the corresponding family of continuous homomorphisms from $A$ to the completions. Taking the closures of the images of $A$ in all possible finite products, it is easy to obtain linking homomorphisms that form a directed system in $\mathsf{C}$. Its limit in $\mathsf{TA}$ obviously satisfies the conditions in the definition of the envelope.
\end{proof}

\begin{rem}
In \cite{Dix}, Dixon considered \emph{varieties} of Banach algebras, which are defined as classes of Banach algebras closed under passing to closed subalgebras, quotient algebras, bounded products and images under isometric isomorphisms. The conditions we impose on a class of Banach algebras are weaker. In particular, it is easy to see that $\mathsf{PG}$ defined in \S\,\ref{s:BanPG} and the class of Banach algebras satisfying a polynomial identity are not closed under bounded products and so they do not form a variety in the sense of Dixon.
\end{rem}

\subsection*{Tensor $\mathsf{CL}$ algebras}
In this section, we fix a class $\mathsf{C}$ of (real or complex) Banach algebras stable  under passing to finite products and closed subalgebras. We consider a left adjoint functor to the forgetful functor from the category $\mathsf{CL}$ to complete locally convex spaces. By analogy with the case of associative algebras and vector spaces, we call the resulting objects  tensor $\mathsf{CL}$ algebras.

\begin{df}\label{Cinften}
Let $E$ be a complete locally convex space. A \emph{tensor $\mathsf{CL}$ algebra} of $E$ is a $\mathsf{CL}$ algebra $T^{\mathsf{C}}(E)$ equipped with a continuous linear map $\mu\! :E\to T^{\mathsf{C}}(E)$ that is the identity component of the left adjoint to the forgetful functor from $\mathsf{CL}$ to the category of complete locally convex spaces.
\end{df}

\begin{rem}
Thus a tensor $\mathsf{CL}$ algebra of $E$ satisfies the following universal property:
for every $\mathsf{CL}$ algebra~$B$ and every continuous linear map
$\psi\!: E \to B$ there is a unique continuous homomorphism
$\widehat\psi\!:T^{\mathsf{C}}(E)\to B$ such that the diagram
\begin{equation*}
  \xymatrix{
E \ar[r]^{\mu}\ar[rd]_{\psi}&T^{\mathsf{C}}(E)\ar@{-->}[d]^{\widehat\psi}\\
 &B\\
 }
\end{equation*}
is commutative.
\end{rem}

Of course, to verify that a tensor $\mathsf{CL}$ algebra exists for each $E$ we need to prove the existence of a left adjoint functor. To do this we use analytic tensor algebras.

If $\mathsf{B}$ is the class of all Banach algebras, then the tensor $\mathsf{BL}$ algebra is called the \emph{analytic tensor algebra} (or \emph{Arens--Michael tensor algebra}) and denoted by $\widehat{T}(E)$. In the case of the ground field~$\mathbb{C}$,  see a proof of the existence and an explicit construction in \cite[\S\,4.2]{Pir_qfree}. For~$\mathbb{R}$ the proof is identical.

\begin{pr}\label{exisCiten}
The tensor $\mathsf{CL}$ algebra exists and is unique up to natural isomorphism for every complete locally convex space.
\end{pr}
\begin{proof}
Let $E$ be a complete locally convex space. Take the analytic tensor algebra $\widehat{T}(E)$ and apply the enveloping functor with respect to $\mathsf{C}$. By Proposition~\ref{exisenvgen}, the envelope exists and is unique up to natural isomorphism. It is easy to see that the resulting algebra satisfies the required universal property.

The uniqueness follows from the definitions.
\end{proof}

\begin{pr}\label{PGLisqu}
Every $\mathsf{CL}$ algebra is a quotient of some tensor $\mathsf{CL}$ algebra.
\end{pr}
\begin{proof}
Let $A\in\mathsf{CL}$. Consider the continuous linear map $\mu\!:A\to T^{\mathsf{C}}(A)$ in Definition~\ref{Cinften} and a continuous homomorphism $\varphi\!:T^{\mathsf{C}}(A)\to A$ induced by the identity. Then $\varphi\mu=1$ and so $\Ker\varphi$ is a closed ideal complemented as a locally convex space. Thus $A\cong T^{\mathsf{C}}(A)/\Ker\varphi$.
\end{proof}

\subsection*{Free $\mathsf{CL}$ algebras}
We also consider a free functor for sets, which in fact is easily reduced to the tensor $\mathsf{CL}$ algebra functor.

\begin{df}\label{freePGL}
Let $X$ be a set. A \emph{free $\mathsf{CL}$ algebra} with generating set~$X$ is a $\mathsf{CL}$ algebra  ${\mathscr F}^{\mathsf{C}}\{X\}$ equipped with a map $\mu\!:X\to {\mathscr F}^{\mathsf{C}}\{X\}$ that satisfies the following universal property: for any $\mathsf{CL}$ algebra~$B$ and any
map $\psi\!: X \to B$ there is a unique continuous homomorphism
$\widehat\psi\!:{\mathscr F}^{\mathsf{C}}\{X\}\to B$ such that the diagram
\begin{equation*}
  \xymatrix{
X \ar[r]^{\mu}\ar[rd]_{\psi}&{\mathscr F}^{\mathsf{C}}\{X\}\ar@{-->}[d]^{\widehat\psi}\\
 &B\\
 }
\end{equation*}
is commutative.
\end{df}
\begin{pr}\label{exisfree}
The free $\mathsf{CL}$  algebra exists and is unique up to natural isomorphism for every set.
\end{pr}
\begin{proof}
It suffices to take the composition of two free functors, free locally convex space functor  and  the tensor $\mathsf{CL}$ algebra functor. The first functor maps every set~$X$ to the vector space having basis of cardinality equal to that of~$X$ and endowed with the strongest locally convex topology and the second is constructed above.
\end{proof}

\section{Banach algebras of polynomial growth and $\mathsf{PGL}$ algebras}
\label{s:BanPG}

\subsection*{Definitions and main properties}
We formulate our main definition in the form that can be applied to both unital and non-unital algebras.
For a real Banach algebra $B$  that does not necessarily have a unit, we use the notation $B_+$ for the algebra with adjoint identity. We also denote the norms on $B$ and $B_+$ by $\|\cdot\|$ and $\|\cdot\|_+$, respectively, and use the same notation for the extensions of these norms to the complexifications.

\begin{lm}\label{PGgendf}
Let $b\in B$. Then the following two conditions are equivalent.

\emph{(1)}~There are $K>0$ and $\alpha\geqslant0$ such that
\begin{equation*}
\|e^{isb}-1\|\leqslant K (1+|s|)^{\alpha} \qquad \text{for all $s\in
\mathbb{R}$.}
\end{equation*}

\emph{(2)}~There are $K>0$ and $\alpha\geqslant0$ such that
\begin{equation*}
\|e^{isb}\|_+\leqslant K (1+|s|)^{\alpha} \qquad \text{for all $s\in
\mathbb{R}$.}
\end{equation*}

If, in addition $B$ is unital, then these conditions hold if and only if
there are $K>0$ and $\alpha\geqslant0$ such that
\begin{equation}\label{eisxs0}
\|e^{isb}\|\leqslant K (1+|s|)^{\alpha} \qquad \text{for all $s\in
\mathbb{R}$.}
\end{equation}
\end{lm}
The proof is straightforward. Note also that the exponent $\alpha$ can be taken the same in all three cases but the constant $K$ can be different.

\begin{df}\label{locPG}

(A)~An element $b$ of a real Banach algebra is of \emph{polynomial growth} if it satisfies the (equivalent) Conditions~(1) and~(2) of Lemma~\ref{PGgendf} (or~\eqref{eisxs0} in the unital case).

(B)~A real Banach algebra  is of \emph{polynomial growth} if all its elements are of polynomial growth.
\end{df}

For unital algebras, the assumption in the preceding definition is a special case of that in~\cite[Definition~2.6]{ArOld}. Note also that  the term `real generalized scalar operator' is sometimes used for an operator satisfying~\eqref{eisxs0}.

\begin{exm}\label{lpPG}
Let $p\in[1,\infty)$. We claim that the real Banach algebra
$\ell_p$ (with pointwise multiplication) is of polynomial growth. It suffices to check that $s\mapsto\|e^{isb}-1\|$ has polynomial growth for every $b\in \ell_p$.  Take $b=(x_n)\in\ell_p$. Since $e^{isb}-1=(e^{isx_n}-1)$, we have
$$
\|e^{isb}-1\|=\left(\sum_{n\in \mathbb{N}} |e^{isx_n}-1|^p\right)^{1/p}.
$$
Note that there is $C>0$ such that
$|e^{iy}-1|\leqslant C|y|$ for every $y\in\mathbb{R}$. Hence,
\begin{equation*}
\|e^{isb}-1\|\leqslant \left(\sum_n C^p|sx_n|^p\right)^{1/p}\leqslant C|s|\,\|b\|\qquad(s\in\mathbb{R}),
\end{equation*}
as desired.
\end{exm}

In what follows we consider mainly unital algebras.

The following structural result is useful. Here we denote by $\Rad B$ the Jacobson radical of an associative algebra~$B$.

\begin{thm}\label{RadPG}
\cite[Theorem~2.8 and  Proposition~2.9]{ArOld}
Let $B$ be a unital Banach algebra of polynomial growth. Then $B/\Rad B$ is commutative and $\Rad B$ is nilpotent.
\end{thm}

\subsection*{$\mathsf{PGL}$ algebras}

In this section, we consider $\mathsf{PGL}$ algebras, that is, $\mathsf{CL}$ algebras in the case when $\mathsf{C}$ is the class of Banach algebras of polynomial growth.

\begin{df}
We denote by $\mathsf{PG}$ the class of real Banach algebras of polynomial growth and by $\mathsf{PGL}$
the class of algebras that are locally in $\mathsf{PG}$. Sometimes instead of `is locally in $\mathsf{PG}$' we say `\emph{is a $\mathsf{PGL}$ algebra}'.
\end{df}

Some results on algebras in this class were obtained in the author's papers~\cite{ArOld,Ar22}.
The following proposition is a reformulation of Proposition 2.11 in~\cite{ArOld}.

\begin{pr}\label{subPG}
An Arens--Michael $\mathbb{R}$-algebra is a $\mathsf{PGL}$ algebra if and only if it is isomorphic to a closed subalgebra of a product of Banach algebras of polynomial growth.
\end{pr}

The following result and especially its corollary are important in subsequent sections.

\begin{thm}\label{SubamaniT}
Every closed subalgebra of a product of  $\mathsf{PGL}$ algebras is a $\mathsf{PGL}$ algebra.
\end{thm}
\begin{proof}
Suppose that $A$ is a closed subalgebra of a product of a family $(A_i)$ of $\mathsf{PGL}$ algebras. It follows from Proposition~\ref{subPG} that each $A_i$ is isomorphic to a closed subalgebra of a product of algebras in $\mathsf{PG}$. It is easy to see (cf. Lemma~\ref{auxtopin} below) that $\prod A_i$ is also isomorphic to a closed subalgebra of a product of algebras in $\mathsf{PG}$. Applying  Proposition~\ref{subPG} in the reverse direction, we conclude
that $A$ is a $\mathsf{PGL}$ algebra.
\end{proof}

For $p\in\mathbb{N}$ denote by ${\mathrm T}_p$ the algebra of upper triangular (including the diagonal) real matrices of order~$p$. When $M$ is a Hausdorff smooth manifold with countable base, we denote by
$C^\infty(M,{\mathrm T}_{p})$ the set of  ${\mathrm T}_{p}$-valued function of class $C^\infty$. It is not hard to see that $C^\infty(M,{\mathrm T}_{p})$, endowed with  the topology of uniform convergence on compact subsets, is an Arens--Michael algebra.
Note that $C^\infty(M,{\mathrm T}_{p})$ can be identified with ${\mathrm T}_{p}(C^\infty(M))$, the algebra of upper triangular  real matrices with entries in ${\mathrm T}_{p}$.

\begin{co}\label{maniTpPGL}
Every closed subalgebra of a product of algebras of the form $C^\infty(M,{\mathrm T}_{p})$, where $M$ is a Hausdorff smooth manifold with countable base and $p\in\mathbb{N}$, is a $\mathsf{PGL}$ algebra.
\end{co}
\begin{proof}
By Theorem~\ref{SubamaniT}, it is sufficient to show that each algebra of the form $C^\infty(M,{\mathrm T}_{p})$ is locally in $\mathsf{PG}$.

Using the gluing property for $C^\infty$-functions, we conclude that each $C^\infty(M,{\mathrm T}_p)$ is topologically isomorphic to a closed subalgebra of a product of algebras of the form $C^\infty(V,{\mathrm T}_{p})$, where each $V$  is an open subset of $\mathbb{R}^m$. (Here $m$ is the dimension of~$M$.) It follows from \cite[Theorem 2.12]{ArOld} that every algebra of the above form is locally in $\mathsf{PG}$ and so is $C^\infty(M,{\mathrm T}_p)$ by Theorem~\ref{SubamaniT}.
\end{proof}

\section{Envelopes of certain commutative and non-commutative algebras}
\label{s:EnvComm}

In general, the envelope of a commutative algebra with respect to $\mathsf{PG}$  can be treated as an algebra of $C^\infty$-functions on some commutative space and, moreover, the envelope of a non-commutative algebra as an algebra of $C^\infty$-functions on a non-commutative space. In this section, we consider first the commutative case and next the following non-commutative examples:
the envelopes of the coordinate algebras of quantum planes, the quantum groups $SL_q(2,\mathbb{R})$ and  the universal enveloping algebras of finite-dimensional Lie algebras. 

\subsection*{Envelopes with respect to $\mathsf{PG}$}

Envelopes with respect to the class of Banach algebras of polynomial growth are the focus of this article.

\begin{pr}\label{exisenv}
The envelope of every locally convex algebra with respect to the class $\mathsf{PG}$ exists and is unique up to natural isomorphism.
\end{pr}
\begin{proof}
The result follows from Proposition~\ref{exisenvgen} because
 $\mathsf{PG}$ is stable  under passing to finite products and closed subalgebras; see \cite[Proposition 2.11]{ArOld}.
\end{proof}

\begin{nota}
For a real locally convex algebra $A$ we denote the envelope with respect to the class $\mathsf{PG}$ by $\widehat A^{\,\mathsf{PG}}$; cf. Definition~\ref{defengen}.
\end{nota}

\subsection*{The commutative case}
Although our main object of study are envelopes of non-commutative algebras, we first look at the commutative case. The following result is the main reason that led the author to study envelopes with respect to~$\mathsf{PG}$.

\begin{pr}\label{envpol}
The natural embedding $\mathbb{R}[x_1,\ldots ,x_n]\to C^\infty(\mathbb{R}^n)$, $n\in\mathbb{N}$, is an envelope with respect to $\mathsf{PG}$.
\end{pr}
\begin{proof}
The assertion is a reformulation of a well-known result on the existence of a $C^\infty$-functional calculus for a finite set of pairwise commuting elements of polynomial growth in a Banach algebra; see, e.g., \cite[Theorem 3.2]{ArOld}.
\end{proof}

\begin{pr}\label{invCinfc}
The natural embedding $\mathbb{R}[x,x^{-1}]\to C^\infty(\mathbb{R}^\times)$, where $\mathbb{R}^\times\!:=\mathbb{R}\setminus \{0\}$, is an envelope with respect to $\mathsf{PG}$.
\end{pr}

For the proof we need the following functional calculus.

\begin{pr}\label{svcommfc}
Let $b$ be an element of polynomial growth in a real Banach algebra~$B$ and $V$ an open subset of $\mathbb{R}$ containing $\Sp b$. Then there exists a unique multiplicative functional calculus, i.e., a continuous homomorphism $\Phi_V\!:C^\infty (V)\to B$ extending the homomorphism $\mathbb{R}[x]\to B\!:x\mapsto b$.
\end{pr}
\begin{proof}
A proof of a similar result for complex-valued functions of class $C^\infty$ and a complex Banach algebra can be found, for example, in \cite[Theorem~1]{Ba79}. In the case of a real Banach algebra, it suffices to apply the result to the complexification of the algebra~$B$ and note that $\Phi_V$ sends all the real-valued functions to~$B$. The latter follows from the fact that $\mathbb{R}[x]$ is dense in $C^\infty(V)$.
\end{proof}

\begin{proof}[Proof of Proposition~\ref{invCinfc}]
Suppose that $B\in\mathsf{PG}$ and $b$ is an invertible element of~$B$. Then $0\notin \Sp b$ and thus $\mathbb{R}^\times$ is an open neighborhood of $\Sp b$. Since  $b$ is of polynomial growth, we can apply Proposition~\ref{svcommfc} with $V=\mathbb{R}^\times$. Hence there is a unique continuous homomorphism $C^\infty(\mathbb{R}^\times)\to B$ extending $\mathbb{R}[x,x^{-1}]\to B$. Since $C^\infty(\mathbb{R}^\times)\in\mathsf{PGL}$ by Corollary~\ref{maniTpPGL}, this means that $\mathbb{R}[x,x^{-1}]\to C^\infty(\mathbb{R}^\times)$ is an envelope with respect to~$\mathsf{PG}$.
\end{proof}

Note that the homomorphism $C^\infty(\mathbb{R}^\times)\to B$ obtained in the proof of Proposition~\ref{invCinfc} can be interpreted as a functional calculus for an invertible element.



\subsection*{First non-commutative example}
We start our discussion of non-commutative algebras with a simple example.

\begin{exm}\label{qudi}
Consider the universal real algebra $A$ with generators $u$ and $v$ satisfying the relation
$vu=1$. Its $C^*$-version  is called the Toeplitz algebra or  the `functional algebra on the quantum disc'; see, e.g., \cite{Da96,Mu97}. Its `holomorphic' \cite{Pa17,Pi21} and `smooth' versions \cite{Cu97,Cu04} are also of interest. The latter is an extension of $C^\infty(S^1)$ with respect to a certain ideal.
Of course, it would be natural to expect that the `smooth Toeplitz algebra' coincides with the envelope of $A$ with respect to $\mathsf{PG}$. But this is not the case.  In fact, the envelope satisfies rigid restrictions  and we get $C^\infty(\mathbb{R}^\times)$.

Indeed, if $vu=1$ holds in a Banach algebra of polynomial growth, then by Theorem~\ref{RadPG} $uv-1$, being a commutator, belongs to the radical and is therefore nilpotent. Since $uv-1$ is an idempotent, it equals~$0$ and so $uv=1$. Then $v=u^{-1}$ and hence $\widehat A^{\,\mathsf{PG}} $ coincides with the envelope of the commutative algebra $\mathbb{R}[u,u^{-1}]$, i.e., $C^\infty(\mathbb{R}^\times)$; see Proposition~\ref{invCinfc}.
\end{exm}

\subsection*{Quantum plane}
Consider the coordinate algebra of the quantum plane over~$\mathbb{R}$, or more precisely, the universal real algebra with generators $x$ and $y$ satisfying the relation $xy = qyx$, where $q\in\mathbb{R}\setminus\{0 \}$. Denote it by $\mathcal{R}(\mathbb{R}^2_q)$.

In the case when $q=1$ we obtain a commutative algebra, $\mathcal{R}(\mathbb{R}^2_1)\cong \mathbb{R}[x,y]$. We have from Proposition~\ref{envpol} that its envelope with respect to $\mathsf{PG}$ is $C^\infty(\mathbb{R}^2)$.

\begin{lm}\label{xynil}
Let $x$ and $y$ be elements of a Banach algebra of polynomial growth and $xy = qyx$ for some $q\in\mathbb{R}\setminus\{0,1\}$. Then  $xy$ is nilpotent.
\end{lm}
\begin{proof}
By the first part of Theorem~\ref{RadPG}, every commutator belongs to the radical. In particular, $xy=(1-q^{-1})^{-1}[x,y]$ is in the radical. Then $xy$ is nilpotent by the second part of Theorem~\ref{RadPG}.
\end{proof}

 Put
$\mathcal{R}(\Omega)\!:=\{(f,g)\in \mathbb{R}[t]\times \mathbb{R}[t]\!:\,f(0)=g(0) \}$ and $C^\infty(\Omega)\!:= \{(f,g)\in C^\infty(\mathbb{R})\times C^\infty(\mathbb{R})\!:\,f(0)=g(0)\}$. Denote $(t,0)$ and $(0,t)$ by $X$ and $Y$ and consider the linear maps
$$
\Phi_x\!:\mathbb{R}[x]\to \mathcal{R}(\Omega)\!: x\mapsto X \quad \text{and}\quad \Phi_y\!:\mathbb{R}[y]\to \mathcal{R}(\Omega)\!: y\mapsto Y.
$$
Then
$$
\{X^i,\,Y^i,\,1;\,\,i\in\mathbb{N}\}
$$
is a linear basis of~$\mathcal{R}(\Omega)$. Hence every element of $\mathcal{R}(\Omega)$ has the form
$\Phi_x(f)+\Phi_y(g)$, where $f\in\mathbb{R}[x]$ and $g\in\mathbb{R}[y]$ with $f(0)=g(0)$.

Put $u\!:=xy$. It is easy to see that the set
$$
\{x^i u^j,\,y^i u^j,\,u^j;\,\, i\in\mathbb{N},\,j\in\mathbb{Z}_+\}
$$
is a linear basis of~$\mathcal{R}(\mathbb{R}^2_q)$ and
the linear map
$$
\mathcal{R}(\mathbb{R}^2_q)\to\mathcal{R}(\Omega)\otimes \mathbb{R}[u]\!:
x^i u^j\mapsto X^i\otimes u^j,\quad y^i u^j\mapsto Y^i\otimes u^j,\quad u^j\mapsto 1\otimes u^j\quad (i\in\mathbb{N},\,j\in\mathbb{Z}_+)
$$
 is an isomorphism. Then we can write every element of $\mathcal{R}(\mathbb{R}^2_q)$ as
\begin{equation}\label{RPhixy}
\sum_{j\in\mathbb{Z}_+}(\Phi_x(f_j)+\Phi_y(g_j))u^j,
\end{equation}
where $f_j\in\mathbb{R}[x]$ and $g_j\in\mathbb{R}[y]$ with $f_j(0)=g_j(0)$.

Taking the tensor product of the natural embeddings $\mathcal{R}(\Omega)\to C^\infty(\Omega)$ and $\mathbb{R}[u]\to\mathbb{R}[[u]]$ (the latter denotes the space of formal power series), we obtain an injective linear map
\begin{equation}\label{ioR2q}
\iota\!:\mathcal{R}(\mathbb{R}^2_q)\to C^\infty(\Omega)\ptn\mathbb{R}[[u]].
\end{equation}
We now show that it can be treated as an enveloping homomorphism.
Denote $C^\infty(\Omega)\ptn\mathbb{R}[[u]]$ by $C^\infty(\mathbb{R}^2_q)$ ($q\in\mathbb{R}\setminus\{0,1\}$). Note that $C^\infty(\mathbb{R}^2_q)$ does not depend on $q$ as a locally convex space but the multiplication introduced in the following theorem does.

\begin{thm}\label{PGqupl}
Let $q\in\mathbb{R}\setminus\{0\}$.

\emph{(A)}~The multiplication determined by the relations $xy = qyx$ and $xy=u$ extends to a continuous operation on $C^\infty(\mathbb{R}^2_q)$ that turns it into a $\mathsf{PGL}$ algebra.

\emph{(B)}~If, in addition, $q\ne 1$, then, taking $C^\infty(\mathbb{R}^2_q)$ with this multiplication, the embedding
$$
\mathcal{R}(\mathbb{R}^2_q)\to C^\infty(\mathbb{R}^2_q)
$$
is an envelope with respect to~$\mathsf {PG}$.
\end{thm}
In \cite{ArOld}, to prove results of this kind the author used a direct approach with estimation of seminorms. Here an indirect approach is applied. It is based on a proposition on closed ideals and submodules combined with a simple lemma on products in the category of complete locally convex spaces. 


\begin{pr}\label{fgclid}
\emph{(A)}~Every ideal of $C^\infty(\mathbb{R}^k)$, $k\in\mathbb{N}$, generated by a finite set of polynomials (or even of real analytic functions) is closed.

\emph{(B)}~Every  $C^\infty(\mathbb{R}^k)$-submodule of the module $C^\infty(\mathbb{R}^k)^p$, $p\in\mathbb{N}$, generated by a finite set of $p$-tuples of polynomials (or even of real analytic functions) is closed.
\end{pr}
It is obvious that Part~(A) follows from Part~(B).  For a proof of Part~(B) see \cite[Chapitre\,VI, p.\,119, Corollaire~1.5]{Tou72}. See also a proof of Part~(A) in \cite[Th\'eor\`eme~4]{Ma60} and partial cases in \cite{Ho58, Lo59}.

Recall that a continuous map is called is topologically injective if it is a homeomorphism onto its range. We need a pair of lemmas on topological injectivity.

\begin{lm}\label{auxtopin}
\emph{(A)}~Let $\{\varphi_i\!:X_i\to Y_i\}_{i\in I}$ be a family of continuous linear maps of complete locally convex spaces and $\varphi$ denote the induced map $\prod_{i\in I} X_i\to \prod_{i\in I} Y_i$.
\begin{itemize}
  \item If each $\varphi_i$ has closed range, then so does $\varphi$.
  \item If each $\varphi_i$ is topologically injective, then so is $\varphi$.
\end{itemize}

\emph{(B)}~Let $\{\psi_i\!:X\to Y_i\}_{i\in I}$ be a family of continuous linear maps of complete locally convex spaces. If  $\psi_i$ is topologically injective for some $i$, then so is the induced map $X\to \prod_{i\in I} Y_i$.
\end{lm}
\begin{proof}
(A)~The first assertion in Part~(A) holds since a product of closed subsets is closed. The second assertion can be derived from the definitions or from the general interchange property of limits, in particular, products and equalizers.

(B)~In the category of complete locally convex spaces, a morphism is an extremal monomorphism if and only if it is a topologically injective map. So Part~(B) follows from general properties of products and extremal monomorphisms; see, e.g., \cite[Propositions 10.15(2) and 10.26(2)]{AHS}.
\end{proof}

Note that a family  $(\psi_{ij}\!:X_i\to Y_j;\,i\in I,\,j\in J)$ of continuous linear maps of complete locally convex spaces induces a continuous linear map $\prod_{i\in I}X_i\to \prod_{j\in J} Y_j$ under the assumption that for every $j$ the subset of $i$ such that $\psi_{ij}\ne 0$ is finite. Indeed, the finiteness condition implies that there is a morphism from $\prod_{i\in I}X_i$ to $ Y_j$  for every $j$ and hence from  $\prod_{i\in I}X_i$ to $\prod_{j\in J} Y_j$.

\begin{lm}\label{auxtopiC}
Let $(\psi_{nm}\!:X_n\to Y_m;\,n,m\in \mathbb{N})$ be a family of continuous linear maps between Fr\'echet spaces such that $\psi_{nm}=0$ whenever $n>m$ (or whenever  $n<m$) and $\psi_{mm}$ is topologically injective for every~$m$. Then the map $\prod X_n\to \prod Y_m$ induced by this family is topologically injective.
\end{lm}
\begin{proof}
Consider the case when $\psi_{nm}=0$ for $n>m$.
It follows from a theorem of Dugundji \cite[Theorem 4.1]{Du51} that a continuous linear map between Fr\'echet spaces is topologically injective if and only if it has a continuous left inverse map (possibly, non-linear).

Note first that $X_1\times X_2\to Y_1\times Y_2$ has a continuous left inverse map.
Indeed, fix  continuous maps left inverse to $\psi_{11}$ and $\psi_{22}$ and denote them by $\psi_{11}^l$ and $\psi_{22}^l$, respectively. It is easy to see that the  continuous map given by
$$
(y_1,y_2)\mapsto (\psi_{11}^l(y_1),\,\psi_{22}^l(y_2-\psi_{12}\psi_{11}^l(y_1)))
$$
is left inverse (cf. the formula for the inverse of a triangular $2\times 2$ matrix) and so we have the topological injectivity.

Further, it easily follows by induction that $\prod_{n\leqslant N}X_n\to \prod_{m\leqslant N} Y_m$ also has a continuous left inverse map for every $N\in\mathbb{N}$. Moreover, the sequence of left inverse maps is compatible with the projections $\prod^{N+1}X_n\to \prod^{N} X_n$ and $\prod^{N+1}Y_m\to \prod^{N} Y_m$.
Since an infinite product is a projective limit of finite products (see, e.g., \cite[Exercise 11B]{AHS}), we have a continuous map  $\prod Y_m\to \prod X_n$ that is, as easy to see, left inverse to $\prod X_n\to \prod Y_m$. Thus the latter map is topologically injective.

The argument for the case when $\psi_{nm}=0$ for $n<m$ is similar.
\end{proof}

We also need the following result on ordered calculus, which also will be used in \S~\ref{sec:efa}.

\begin{thm}
\label{fucani}
\cite[Theorem~3.3]{ArOld}
Let $b_1, \dots, b_m$ be elements of an Arens--Michael $\mathbb{R}$-algebra~$B$. Suppose that $b_1,\dots,b_k$ \emph{(}$k\leqslant m$\emph{)} are of polynomial growth and $b_{k+1},\dots,b_m$ are nilpotent. Then the linear map $\mathbb{R}[\lambda_1,\dots,\lambda_m]\to B$ taking the (commutative) monomial $\lambda_1^{\beta_1}\cdots \lambda_m^{\beta_m}$ to the (non-commutative) monomial $b_1^{\beta_1}\cdots b_m^{\beta_m}$ extends to a continuous linear map
$$
C^\infty(\mathbb{R}^k)\mathbin{\widehat{\otimes}} \mathbb{R}[[\lambda_{k+1},\dots,\lambda_m]]\to B.
$$
\end{thm}

This result is a strengthening of a well-known theorem on ordered functional calculus; see a discussion in \cite{ArOld}.

\begin{proof}[Proof of Theorem~\ref{PGqupl}]
(A)~It follows from Corollary~\ref{maniTpPGL} that it suffices to construct a homomorphism from $\mathcal{R}(\mathbb{R}^2_q)$ to a product of algebras of the form $C^\infty(M,{\mathrm T}_{p})$ and extend it to a topologically injective continuous linear map defined on $C^\infty(\Omega)\ptn\mathbb{R}[[u]]$. This basic idea is the same as in the proof of Theorem~4.3 in~\cite{ArOld} and we also use it  in the proofs of Theorems~\ref{PGSL2R} and~\ref{multCiffk}.

For $p\in\mathbb{N}$ take the following triangular matrices of order~$p$:
\begin{equation}\label{defKnEn}
K_p\!:=
\begin{pmatrix}
 q^{p-1} &0 &\ldots&0&0\\
 0 &q^{p-2} &\ldots&0&0\\
 \vdots &\ddots &\ddots&\ddots&\vdots\\
 0 &0&\ldots&q&0\\
 0 &0 &\ldots&0&1\\
   \end{pmatrix}
,\quad
E_p\!:=
\begin{pmatrix}
 0 &1 &0&\ldots&0\\
 0 &0 &1 &\ldots&0\\
 \vdots &\ddots &\ddots&\ddots&\vdots\\
 0 &0 &\ddots&\ddots&1\\
 0 &0 &\ldots&0&0\\
   \end{pmatrix}.
\end{equation}
Since $K_pE_p=qE_pK_p$, the formulas
$$
\pi_{p,\lambda}\!:x\mapsto \lambda K_p,\,y\mapsto E_p\quad\text{and}\quad \pi'_{p,\mu}\!:x\mapsto E_p,\,y\mapsto \mu K_p^{-1}\qquad(\lambda,\mu\in\mathbb{R})
$$
define homomorphisms from $\mathcal{R}(\mathbb{R}^2_q)$ to ${\mathrm T}_p$. Treating $\lambda,\mu$ as variables, we obtain homomorphisms $\mathcal{R}(\mathbb{R}^2_q)\to{\mathrm T}_p(C^\infty(\mathbb{R}))$ and denote them by $\widetilde\pi_p$ and $\widetilde\pi'_p$, respectively.

Restrict $\widetilde\pi_p$ and $\widetilde\pi'_p$ to $\mathbb{R}[x]\otimes\mathbb{R}[u]$ and $\mathbb{R}[y]\otimes\mathbb{R}[u]$, respectively. Since $E_p$ is nilpotent, it follows from Theorem~\ref{fucani} that the restrictions can be extended to continuous linear maps $C^\infty(\mathbb{R})\ptn\mathbb{R}[[u]]\to{\mathrm T}_p(C^\infty(\mathbb{R}))$.

We can identify $(C^\infty(\mathbb{R})\ptn\mathbb{R}[[u]])\times(C^\infty(\mathbb{R})\ptn\mathbb{R}[[u]])$ with a power of $C^\infty(\mathbb{R})\times C^\infty(\mathbb{R})$.
Consider the continuous linear map
\begin{equation}\label{2CR2TCR}
\rho\!:C^\infty(\Omega)^{\mathbb{Z}_+}\to\prod_{p=1}^\infty ({\mathrm T}_p(C^\infty(\mathbb{R}))\times {\mathrm T}_p(C^\infty(\mathbb{R})))
\end{equation}
induced by the sequence $(\widetilde\pi_p,\widetilde\pi_p')$. (Recall that $C^\infty(\Omega)$ is a closed subspace of $C^\infty(\mathbb{R})\times C^\infty(\mathbb{R})$.) Since  the restriction of $\rho$ to $\mathcal{R}(\mathbb{R}^2_q)$ is a homomorphism, to complete the proof we show that $\rho$ is topologically injective.

We now select only those copies that correspond to upper right entries of the matrices on the right-hand side of~\eqref{2CR2TCR} and get a map
$$
\rho'\!:C^\infty(\Omega)^{\mathbb{Z}_+}\to (C^\infty(\mathbb{R})\times C^\infty(\mathbb{R}))^{\mathbb{N}}.
$$
Since $(C^\infty(\mathbb{R})\times C^\infty(\mathbb{R}))^{\mathbb{N}}$ is a direct factor of the space  on the right-hand side of~\eqref{2CR2TCR},  Part~(B) of Lemma~\ref{auxtopin} implies that to prove that $\rho$ is topologically injective it suffices to show that so is $\rho'$.

Further, we find an explicit form of $\rho'$. Note that
$$
\widetilde\pi_{p,\lambda}(u)=\lambda  K_pE_p\quad\text{and}\quad\widetilde\pi'_{p,\mu}(u)=\mu  E_pK_p^{-1},
$$
or, in a detailed form,
$$
\widetilde\pi_{p,\lambda}(u)=\lambda
\begin{pmatrix}
 0 &q^{p-1} &0&\ldots&0\\
 0 &0 &q^{p-2}&\ldots&0\\
 \vdots &\ddots &\ddots&\ddots&\vdots\\
 0 &0 &\ddots&\ddots&q\\
 0 &0 &\ldots&0&0\\
   \end{pmatrix},
   \qquad\widetilde\pi'_{p,\mu}(u)=\mu
\begin{pmatrix}
 0 &q^{2-p} &0&\ldots&0\\
 0 &0 &q^{3-p}&\ldots&0\\
 \vdots &\ddots &\ddots&\ddots&\vdots\\
 0 &0 &\ddots&\ddots&1\\
 0 &0 &\ldots&0&0\\
   \end{pmatrix}.
$$

Let an element $a$ of $\mathcal{R}(\mathbb{R}^2_q)$  be written  as in~\eqref{RPhixy}.
To find the upper right entries of $\widetilde\pi_{p,\lambda}(a)$ and $\widetilde\pi'_{p,\mu}(a)$ note that 
$\widetilde\pi_{p,\lambda}(u)^k$  and $\widetilde\pi'_{p,\mu}(u)^k$ vanish when $k\geqslant p$. Considering the case where $r\leqslant p-1$, we see that the only one non-trivial entry in the right column of $\widetilde\pi_{p,\lambda}(u)^k$ and $\widetilde\pi'_{p,\mu}(u)^k$ equals respectively $\lambda^k q^{k(k+1)/2}$ and
$\mu^k q^{-k(k-1)/2}$ both at the $(p-1-k)$th place.

 Note that  $\widetilde\pi_{p,\lambda}(\Phi_x(f_j))$ is the diagonal matrix with the entries $f_j(\lambda q^{p-1}),\ldots,f_j(\lambda)$ and 
$$
\widetilde\pi_{p,\lambda}(\Phi_y(g_j))=
\begin{pmatrix}
 g_j(0) &g_j'(0) &0&\ldots&g_j^{(p-1)}(0)/(p-1)!\\
 0 &g_j(0)&g_j'(0)&\ldots&0\\
 \vdots &\ddots &\ddots&\ddots&\vdots\\
 0 &0 &\ddots&\ddots&g_j'(0)\\
 0 &0 &\ldots&0&g_j(0)
   \end{pmatrix}.
$$
Therefore the upper right entry of $\widetilde\pi_{p,\lambda}(a)$ is equal to
\begin{equation}\label{urepi}
f_{p-1}(\lambda q^{p-1})\lambda^{p-1} q^{p(p-1)/2}+\sum_{k=0}^{p-1} \frac{g^{(p-1-k)}_k(0)}{(p-1-k)!}\lambda^k q^{k(k+1)/2}.
\end{equation}
Similarly, we have that the upper right entry of $\widetilde\pi'_{p,\mu}(a)$ is equal to
\begin{equation}\label{urepip}
g_{p-1}(\mu q^{1-p})\mu^{p-1} q^{-(p-1)(p-2)/2}+\sum_{k=0}^{p-1} \frac{f^{(p-1-k)}_k(0)}{(p-1-k)!}\mu^k q^{-k(k-1)/2}.
\end{equation}

By continuity, \eqref{urepi} and \eqref{urepip} hold for every element of $C^\infty(\mathbb{R}^2_q)$.
We can treat $\rho'$ as the map induced by a family $(\psi_{p-1,k} : X_{p-1}\to Y_k; p,k\in\mathbb{N})$ of continuous linear maps, where $X_{p-1}=C^\infty(\Omega)$ and $Y_k=C^\infty(\mathbb{R})\times C^\infty(\mathbb{R})$. Using the fact that $f_{p-1}(0)=g_{p-1}(0)$, we conclude that $ \psi_{p-1,p}$  maps $((f_{p-1},g_{p-1}))$ to
$$
(\lambda\mapsto (f_{p-1}(\lambda q^{p-1})+f_{p-1}(0))\lambda^{p-1} q^{p(p-1)/2},\,\mu\mapsto (g_{p-1}(\mu q^{1-p})+g_{p-1}(0))\mu^{p-1} q^{-(p-1)(p-2)/2} ).
 $$
Also, $\psi_{p-1,k}=0$ when $k>p-1$. Lemma~\ref{auxtopiC} implies that to prove that $\rho'$ is topologically injective it suffices to prove that each of $\psi_{p-1,p}$  is topologically injective. We can assume that $\psi_{p-1,p}$  is defined on the whole $C^\infty(\mathbb{R})\times C^\infty(\mathbb{R})$. Thus it is sufficient to check the topological injectivity of two maps on $C^\infty(\mathbb{R})$. 

Each of them is a composition of three linear maps, the first is  induced by a diffeomorphism, the second is a shift by a one-dimension operator and the third  is the multiplication operator on a polynomial. Each of these maps has closed range --- the first and second  maps because they are topological isomorphisms and the third  map by Proposition~\ref{fgclid}. Also, it is easy to see that all three maps are injective. Thus, by the inverse mapping theorem for Fr\'echet spaces, they are topologically injective and thus this completes the proof of Part~(A).

(B)~Suppose now that $q\ne1$. Let $B$ be a Banach algebra of polynomial growth and  $\varphi\!:\mathcal{R}(\mathbb{R}^2_q)\to B$  a homomorphism. Denote by $\theta_x$ and $\theta_y$ the $C^\infty$-functional calculi  (i.e., continuous homomorphisms from $C^\infty(\mathbb{R})$ to $B$) corresponding to $\varphi(x)$ and $\varphi(y)$. Recall that $C^\infty(\Omega)=\{(f,g)\in C^\infty(\mathbb{R})\times C^\infty(\mathbb{R})\!:\,f(0)=g(0)\}$ by definition. Consider the continuous linear map
$$
C^\infty(\Omega)\to B\!: (f,g)\mapsto \theta_x(f)+\theta_y(g)-f(0).
$$
Note that $\varphi(u)$ is nilpotent by Lemma~\ref{xynil}. Then $\mathbb{R}[u]\to B\!:u\mapsto \varphi(u)$ can be extended to continuous linear map $\mathbb{R}[[u]]\to B$ (this is a partial case of Theorem~\ref{fucani}). Take the composition of the tensor product of these two maps and the linearization of the multiplication in~$B$,
$$
\widehat\varphi\!:C^\infty(\Omega)\ptn\mathbb{R}[[u]]\to B\ptn B\to B.
$$
It is easy to see that $\widehat\varphi\iota=\varphi$, where $\iota$ is defined in~\eqref{ioR2q}. It follows from Part~(A) that $\iota$ is a homomorphism and $\widehat\varphi$ is a continuous homomorphism. Since $\mathcal{R}(\mathbb{R}^2_q)$ is dense in $C^\infty(\Omega)\ptn\mathbb{R}[[u]]$, such a homomorphism is unique. Thus, the universal property in Remark~\ref{simpdef} holds and so $\mathcal{R}(\mathbb{R}^2_q)\to C^\infty(\Omega)\ptn\mathbb{R}[[u]]$ is an envelope with respect to~$\mathsf{PG}$.
\end{proof}

\subsection*{Quantum group $SL_q(2,\mathbb{R})$}

Consider the coordinate algebra of the quantum group $SL_q(2,\mathbb{R})$. Namely, for $q\in\mathbb{R}\setminus\{0\}$
denote by ${\mathcal R}(SL_q(2,\mathbb{R}))$ the universal real associative algebra with generators $a$, $b$, $c$, $d$ and relations
\begin{equation}
\label{SL2qrel1}
ab = qba,\qquad ac = qca, \qquad bc = cb,
\end{equation}
\begin{equation}
\label{SL2qrel2}
  bd = qdb, \qquad cd = qdc,
\end{equation}
\begin{equation}
\label{SL2qrel3}
da- q^{-1}bc=1 , \qquad ad - qbc=1;
\end{equation}
see, e.g., \cite[\S\,4.1.2]{KSc}.

In what follows we restrict ourselves to the additional assumption that $q$ is not a root of unity, i.e., $q\ne\pm1$. We show first that in a  Banach algebra of polynomial growth the above relations can be simplified.

\begin{lm}\label{adinbcnil}
Let $q\ne\pm1$ and $a$, $b$, $c$ and $d$ be elements of a  Banach algebra of polynomial growth satisfying the relations~\eqref{SL2qrel1}--\eqref{SL2qrel3}. Then

\emph{(A)}~$a$ and $d$ are invertible;

\emph{(B)}~$b$ and $c$ are nilpotent.
\end{lm}
\begin{proof}
(A)~By Theorem~\ref{RadPG}, every commutator belongs to the radical. In particular, $bc=(q-q^{-1})^{-1} [a,d]$ and so belongs to the radical; cf. Lemma~\ref{xynil}.
Then $bc$ is nilpotent by Theorem~\ref{RadPG}. Hence $1+q^{-1}bc$ and $1+qbc$ are invertible. It follows from~\eqref{SL2qrel3} that $a$ and $d$ are also invertible.

(B)~It is easy to check by induction that $(ba)^n a^{-n}=q^{n(n-1)/2}b^n$ for every $n\in\mathbb{N}$. By Lemma~\ref{xynil}, $ba$ is  nilpotent and hence so is $b$. Similarly, we obtain that $c$ is also nilpotent.
\end{proof}

The following lemma is obtained by a simple calculation.

\begin{lm}\label{onlyabc}
Let $a$, $b$ and $c$ be elements of some real algebra and satisfy the relations in~\eqref{SL2qrel1}. If $a$ is invertible and $d\!: =a^{-1}(1+ qbc)$, then the relations in~\eqref{SL2qrel2} and~\eqref{SL2qrel3} also hold.
\end{lm}

Denote by $A_q$ the universal real associative algebra with generators $a$, $a^{-1}$, $b$, $c$ and the relations in~\eqref{SL2qrel1}. Combining Lemmas~\ref{adinbcnil} and~\ref{onlyabc} we get the following assertion.

\begin{co}\label{isoAq}
The obviously defined homomorphism ${\mathcal R}(SL_q(2,\mathbb{R}))\to A_q$ induces a topological isomorphism
${{\mathcal R}(SL_q(2,\mathbb{R}))\sphat}{\,\,}^{\,\mathsf{PG}}\to \widehat A_q^{\,\mathsf{PG}}$.
\end{co}

The relations in \eqref{SL2qrel1} easily imply that
$$
\{a^ib^jc^k;\quad i\in\mathbb{Z},\,j,k\in\mathbb{Z}_+\}
$$
is a PBW basis in $A_q$. Thus we immediately obtain the following result.

\begin{pr}
The obviously defined linear map
$$
\mathbb{R}[a,a^{-1}]\otimes \mathbb{R}[b,c]\to A_q
$$
is an isomorphism of vector spaces.
\end{pr}

Taking the composition of the inverse to the previous map and the tensor product of natural the embeddings $\mathbb{R}[a,a^{-1}]\to C^\infty(\mathbb{R}^\times)$ and $\mathbb{R}[b,c]\to\mathbb{R}[[b,c]]$ we obtain an injective linear map
\begin{equation}\label{ioSL2q}
A_q\to C^\infty(\mathbb{R}^\times)\ptn \mathbb{R}[[b,c]].
\end{equation}
Denote $C^\infty(\mathbb{R}^\times)\ptn \mathbb{R}[[b,c]]$ by $C^\infty(SL_q(2,\mathbb{R}))$. (To show that this space can be endowed with a multiplication that depends on $q$, we again use the index $q$.)

\begin{thm}\label{PGSL2R}
Let $q\in\mathbb{R}\setminus\{0,1,-1\}$.

\emph{(A)}~The multiplication in $A_q$ determined by the relations in \eqref{SL2qrel1} extends to a continuous operation on $C^\infty(SL_q(2,\mathbb{R}))$ that turns it into a $\mathsf{PGL}$ algebra.

\emph{(B)}~Taking $C^\infty(SL_q(2,\mathbb{R}))$  with this multiplication, the embedding
$$
{\mathcal R}(SL_q(2,\mathbb{R}))\to C^\infty(SL_q(2,\mathbb{R}))
$$
is an envelope with respect to~$\mathsf {PG}$.
\end{thm}

The first step of the proof is the following lemma.
\begin{lm}\label{Aqext}
Every homomorphism from $A_q$ to a $\mathsf{PGL}$ algebra extends to a continuous linear map from $C^\infty(SL_q(2,\mathbb{R}))$.
\end{lm}
\begin{proof}
Let $B$ be a $\mathsf{PGL}$ algebra and  $\varphi\!:A_q\to B$ a homomorphism. We can assume that $B$ is a Banach algebra of polynomial growth. Denote by $\varphi_a$ and $\varphi_{b,c}$  the corresponding homomorphisms from $\mathbb{R}[a,a^{-1}]$ and $\mathbb{R}[b,c]$ to $B$.

By Proposition~\ref{invCinfc}, $\varphi_a$ extends to a continuous homomorphism $C^\infty(\mathbb{R}^\times)\to B$. On the other hand,
Lemma~\ref{adinbcnil} implies that $\varphi(b)$ and $\varphi(c)$ are nilpotent. So by Theorem~\ref{fucani}, the linear map $\varphi_{b,c}\!:\mathbb{R}[b,c]\to B\!:b^ic^j\mapsto \varphi(b^ic^j)$ extends to a continuous linear map $\mathbb{R}[[b,c]]\to B$. (Actually, this map is a homomorphism but we do not need this fact here.)
The composition
$$
C^\infty(\mathbb{R}^\times)\ptn \mathbb{R}[[b,c]]\to B\ptn B\to B
$$
of the tensor product of these two maps and the linearization of the multiplication in~$B$ is the desired extension.
\end{proof}

\begin{proof}[Proof of Theorem~\ref{PGSL2R}]
(A)~By  Corollary~\ref{maniTpPGL}, it suffices to construct a homomorphism from $A_q$ to a product of algebras of the form $C^\infty(M,{\mathrm T}_{p})$ that extends to a topologically injective continuous linear map defined on $C^\infty(SL_q(2,\mathbb{R}))$.

We again use the matrices $E_p$ and $K_p$ introduced in~\eqref{defKnEn}. It is easy to see from the relation $K_pE_p=qE_pK_p$ that for $p\in\mathbb{N}$, $\lambda\in\mathbb{R}^\times$ and $\mu,\nu\in\mathbb{R}$ the correspondence
$$
\pi_{p,\lambda,\mu,\nu}\!:a\mapsto \lambda K_p,\,b\mapsto \mu E_p,\,c\mapsto \nu E_p
$$
determines a representation of~$A_q$.

Fix $p$ temporarily.
Write every element of $A_q$ as $x=\sum_{i,j\in\mathbb{Z}_+} f_{ij}(a)b^jc^j$, where $f_{ij}$ are Laurent polynomials. Then $\pi_{p,\lambda,\mu,\nu}(x)$ equals
\begin{equation}\label{SLmat}
\begin{pmatrix}
 f_{00}(\lambda q^{p-1}) &f_{10}(\lambda q^{p-1})\mu+ f_{01}(\lambda q^{p-1})\nu   &&\ldots&\sum_{i+j=p-1} f_{ij}(\lambda q^{p-1})\mu^i\nu^j \\
 0 &f_{00}(\lambda q^{p-2})&\ddots&\ldots&\\
 \vdots &\ddots &\ddots&\ddots&\vdots\\
 0 &0 &\ddots&f_{00}(\lambda q)& f_{10}(\lambda q)\mu+ f_{01}(\lambda q)\nu \\
 0 &0 &\ldots&0&f_{00}(\lambda)\\
   \end{pmatrix}.
\end{equation}

Treating $\lambda,\mu,\nu$ as variables, we have a homomorphism $A_q\to{\mathrm T}_p(C^\infty(\mathbb{R}^\times\times \mathbb{R}^2))$, which we denote by $\widetilde\pi_p$. Since ${\mathrm T}_p(C^\infty(\mathbb{R}^\times\times \mathbb{R}^2))$ is a $\mathsf{PGL}$ algebra by Corollary~\ref{maniTpPGL}, it follows from
Lemma~\ref{Aqext} that $\widetilde\pi_p$  extends to a continuous linear map
$$
C^\infty(SL_q(2,\mathbb{R}))\to {\mathrm T}_p(C^\infty(\mathbb{R}^\times\times \mathbb{R}^2)).
$$

We can identify $C^\infty(\mathbb{R}^\times)\ptn \mathbb{R}[[b,c]]$ with a power of $C^\infty(\mathbb{R}^\times)$, where the copies labelled by $\mathbb{Z}_+^2$. Then $\widetilde\pi_p$ maps a double sequence $(f_{ij})$ to the matrix as in \eqref{SLmat}. Partition $\mathbb{Z}_+^2$ into the union of subsets given by the condition $i+j=r$ and write $C^\infty(\mathbb{R}^\times)\ptn \mathbb{R}[[b,c]]$ as $\prod_{r=0}^\infty  C^\infty(\mathbb{R}^\times)^r$. Consider the  map
$$
\rho\!:\prod_{r=0}^\infty C^\infty(\mathbb{R}^\times)^r\to \prod_{p=1}^\infty {\mathrm T}_p(C^\infty(\mathbb{R}^\times\times \mathbb{R}^2))
$$
induced by the sequence  $(\widetilde\pi_p)$.
Identifying ${\mathrm T}_p(C^\infty(\mathbb{R}^\times\times \mathbb{R}^2))$ with a finite power of $C^\infty(\mathbb{R}^\times\times \mathbb{R}^2)$, we can write the space on the right-hand side as a power of $C^\infty(\mathbb{R}^\times\times \mathbb{R}^2)$. We now select only those copies that correspond to upper right entries of the matrices. By Part~(B) of Lemma~\ref{auxtopin}, to show that $\rho$ is topologically injective it suffices to check that the map
\begin{equation}\label{SLpartrho}
\prod_r C^\infty(\mathbb{R}^\times)^r\to C^\infty(\mathbb{R}^\times\times \mathbb{R}^2)^{\mathbb{Z}_+}\!:(f_{ij})_{i+j=r}\mapsto ((\lambda,\mu,\nu)\mapsto \sum_{i+j+1=p} f_{ij}(\lambda q^{p-1})\mu^i\nu^j)_p
\end{equation}
is topologically injective.

Note that the map $C^\infty(\mathbb{R}^\times)^r\to C^\infty(\mathbb{R}^\times\times \mathbb{R}^2)$ that takes the tuple $(f_{ij})_{i+j=r}$ to the function $(\lambda,\mu,\nu)\mapsto \sum f_{ij}(\lambda)\mu^i\nu^j$ is obviously injective.  Moreover, it has closed range by Proposition~\ref{fgclid} and hence it is
topologically injective by the inverse mapping theorem for Fr\'echet spaces. Since $\lambda\to\lambda q^{p-1}$ is a diffeomorphism of $\mathbb{R}^\times$, each factor in~\eqref{SLpartrho} is also topologically injective and hence so is the product of the maps by Part~(A) of Lemma~\ref{auxtopin}. Thus $\rho$ is in turn topologically injective.

Note finally that the restriction of $\rho$ to $A_q$ is a homomorphism and this completes the proof.

(B)~Corollary~\ref{isoAq} implies that we can replace ${\mathcal R}(SL_q(2,\mathbb{R}))$ by $A_q$. Let $B$ be a Banach algebra of polynomial growth and  $\varphi\!:A_q\to B$ a homomorphism. By Lemma~\ref{Aqext}, $\varphi$ extends to a continuous linear map
$\widehat\varphi\!:C^\infty(SL_q(2,\mathbb{R}))\to B$. It follows from Part~(A) that $\iota$ in~\eqref{ioSL2q} is a homomorphism and $\widehat\varphi$ is a continuous homomorphism with respect to the multiplication extended from $A_q$. Since $A_q$ is dense in $C^\infty(SL_q(2,\mathbb{R}))$, such a homomorphism is unique. Thus, the universal property in Remark~\ref{simpdef} holds and so $A_q\to C^\infty(SL_q(2,\mathbb{R}))$ is an envelope with respect to~$\mathsf{PG}$.
\end{proof}

\subsection*{Envelopes of universal enveloping algebras of finite-dimensional Lie algebras}
We begin with formulating a result from \cite{ArOld} that was the starting point for writing this article. Recall that a finite-dimensional real Lie algebra $\mathfrak{g}$ is said to be \emph{triangular} if it is solvable and for every $x\in\mathfrak{g}$ all the eigenvalues of the linear operator~$\ad x$ belong to~$\mathbb{R}$.

Let $\mathfrak{g}$ be a triangular finite-dimensional real Lie algebra. Fix  a linear basis $e_{k+1},\ldots, e_m$ in~$[\mathfrak{g},\mathfrak{g}]$ and its complement $e_1,\ldots, e_k$ up to a linear basis in~$\mathfrak{g}$ and consider the Fr\'echet space
\begin{equation}\label{Cinffgdef}
C^\infty_\mathfrak{g}\!:=C^\infty(\mathbb{R}^k)\ptn \mathbb{R}[[e_{k+1},\ldots,e_m]].
\end{equation}
Consider the corresponding Poincar\'{e}--Birkoff--Witt basis $\{e^\alpha\!:=e_1^{\alpha_1}\cdots e_m^{\alpha_m}\!:\,\alpha\in \mathbb{Z}_+^m\}$ in $U (\mathfrak{g})$ and the linear map $U(\mathfrak{g})\to C^\infty_\mathfrak{g}$ with dense image given by identifying $e_1,\ldots, e_k$ with the coordinate functions on~$\mathbb{R}^k$ and the embedding $\mathbb{R}[e_{k+1},\ldots,e_m]\to \mathbb{R}[[e_{k+1},\ldots,e_m]]$. The following result combines Theorems~4.3 and~4.4 in~\cite{ArOld}.

\begin{thm}\label{PGtrifdL}
Let $\mathfrak{g}$ be a triangular finite-dimensional real Lie algebra.

\emph{(A)}~The multiplication in $U(\mathfrak{g})$ extends to a continuous operation on $C^\infty_\mathfrak{g}$, which makes it a $\mathsf{PGL}$ algebra,

\emph{(B)}~Taking $C^\infty_ \mathfrak{g}$  with this multiplication, the embedding $U(\mathfrak{g})\to C^\infty_ \mathfrak{g}$ is an envelope with respect to~$\mathsf{PG}$.

\emph{(C)}~The algebra $C^\infty_\mathfrak{g}$ is independent of the choice of a basis in $[\mathfrak{g},\mathfrak{g}]$ and its complement to a basis in~$\mathfrak{g}$.
\end{thm}
The main idea of the proof is the same as in Theorems~\ref{PGqupl} and~\ref{PGSL2R} but the reasoning is cumbersome and contains many technicalities; for details see \cite{ArOld}.

We now turn to arbitrary finite-dimensional real Lie algebras.

\begin{pr}
Every finite-dimensional real Lie algebra has a maximal triangular quotient algebra.
\end{pr}
\begin{proof}
Let $\mathfrak{g}$ be a finite-dimensional real Lie algebra. It suffices to show that if $\mathfrak{h}_1$ and $\mathfrak{h}_2$ are ideals in $\mathfrak{g}$ such that $\mathfrak{g}/\mathfrak{h}_1$ and $\mathfrak{g}/\mathfrak{h}_2$ are triangular, then so is $\mathfrak{g}/(\mathfrak{h}_1 \cap \mathfrak{h}_2)$.

It is obvious that the homomorphism $\mathfrak{g}\to \mathfrak{g}/\mathfrak{h}_1\oplus \mathfrak{g}/\mathfrak{h}_2$ factors through $\mathfrak{g}/(\mathfrak{h}_1 \cap \mathfrak{h}_2)$ and the corresponding map is injective. Thus $\mathfrak{g}/(\mathfrak{h}_1 \cap \mathfrak{h}_2)$ is isomorphic to a subalgebra in $\mathfrak{g}/\mathfrak{h}_1\oplus \mathfrak{g}/\mathfrak{h}_2$. It is easy to see that each subalgebra of a triangular Lie algebra is triangular. Since $\mathfrak{g}/\mathfrak{h}_1\oplus \mathfrak{g}/\mathfrak{h}_2$  is triangular, so is $\mathfrak{g}/(\mathfrak{h}_1 \cap \mathfrak{h}_2)$.
\end{proof}

Denote the maximal triangular quotient algebra of a finite-dimensional real Lie algebra by $\mathfrak{g}^{\mathbf{tri}}$ and consider the composition $U(\mathfrak{g})\to U(\mathfrak{g}^{\mathbf{tri}})\to C^\infty_{\mathfrak{g}^{\mathbf{tri}}}$.

\begin{thm}\label{envhtri}
Let $\mathfrak{g}$ be a finite-dimensional real Lie algebra. Then $U(\mathfrak{g})\to C^\infty_{\mathfrak{g}^{\mathbf{tri}}}$ is an envelope with respect to $\mathsf{PG}$.
\end{thm}
\begin{proof}
Let $B$ be a Banach algebra of polynomial growth and $\varphi\!: U(\mathfrak{g}) \to B$ a continuous homomorphism. By \cite[Proposition 4.1]{ArOld}, $\varphi(\mathfrak{g})$ is a triangular Lie algebra. Therefore $\mathfrak{g}\to B$ factors through $\mathfrak{g}^{\mathbf{tri}}$ and hence $\varphi$ factors through $U(\mathfrak{g}^{\mathbf{tri}})$. Since $\mathfrak{g}^{\mathbf{tri}}$ is triangular, it remains to apply Theorem~\ref{PGtrifdL}.
\end{proof}

\begin{exms}\label{nontri}
(A)~It is easy to see that  $\mathfrak{g}^{\mathbf{tri}}=0$ when $\mathfrak{g}=\mathfrak{s}\mathfrak{l}_2$ or $\mathfrak{g}=\mathbb{R}^2\rtimes\mathfrak{s}\mathfrak{l}_2$. Thus $\widehat U(\mathfrak{g})^{\,\mathsf{PG}}$ is also trivial in both cases. In general $\widehat U(\mathfrak{g})^{\,\mathsf{PG}}=0$  for every semi-simple real Lie algebra~$\mathfrak{g}$.

(B)~Let  $\mathfrak{e}_2$ denote the Lie algebra of the group of motions of the plane~$\mathbb{R}^2$. It has a linear basis $x_1,x_2,x_3$ that satisfies the relations
$$
[x_1,x_2]=x_3, \quad [x_1,x_3]=-x_2, \quad [x_2,x_3]=0.
$$
It is the simplest example of a solvable, but not triangular, Lie algebra (the operator $\ad x_1$ has eigenvalues $i,0,-i$). It is easy to see that $[\mathfrak{e}_2,\mathfrak{e}_2]$ is a minimal ideal with basis $x_2,x_3$. Thus $\mathfrak{e}_2^{\mathbf{tri}}\cong \mathfrak{e}_2/[\mathfrak{e}_2,\mathfrak{e}_2]$. Then $\widehat U(\mathfrak{e}_2)^{\,\mathsf{PG}}$ is isomorphic to $C^\infty(\mathbb{R})$ by Theorem~\ref{envhtri}.
\end{exms}

Recall that ${\mathcal R}(SL_q(2,\mathbb{R}))$ and $U(\mathfrak{g})$ have Hopf algebra structures. Similarly,  $\widehat U(\mathfrak{g})^{\,\mathsf{PG}}$ and ${{\mathcal R}(SL_q(2,\mathbb{R}))\sphat}{\,\,}^{\,\mathsf{PG}}$ admit induced topological Hopf algebra structures. This topic is discussed in~\cite{ArNew2}.

\section{Envelopes of free algebras}
\label{sec:efa}

When $\mathsf{C}=\mathsf{PG}$, tensor $\mathsf{CL}$ algebras can be naturally called \emph{tensor $\mathsf{PGL}$ algebras} and, particularly, free $\mathsf{CL}$ algebras can be called \emph{free $\mathsf{PGL}$ algebras}. We also name them \emph{$C^\infty$-tensor algebras}  and  \emph{$C^\infty$-free algebras}. In this section we describe an explicit form of the free $\mathsf{PGL}$ algebra ${\mathscr F}^{\,\mathsf{PG}}\{X\}$ in the case when $X$ is a finite set of cardinality~$k$. For brevity, we denote ${\mathscr F}^{\,\mathsf{PG}}\{X\}$ by~${\mathscr F}_{k}^{\,\mathsf{PG}}$. The proof is given only for $k\leqslant 2$; the general case is postponed to a forthcoming paper.

\subsection*{Formal tensor algebra}
We need another version of tensor algebra --- formal tensor algebra; cf. \cite{DM77} in the Banach space case.

\begin{df}\label{formtena}
Let $E$ be a complete (real or complex) locally convex space. The \emph{formal tensor algebra} of $E$ is an Arens--Michael algebra $[T](E)$ with a continuous linear map $\mu\!:E\to [T](E)$ satisfying the following condition.
For every Banach algebra~$B$ and every continuous linear map $\psi\!: E \to B$ with nilpotent range there is a unique unital continuous homomorphism $\widehat\psi\!:[T](E)\to B$ such that the diagram
\begin{equation*}
  \xymatrix{
E \ar[r]^{\mu}\ar[rd]_{\psi}&[T](E)\ar@{-->}[d]^{\widehat\psi}\\
 &B\\
 }
\end{equation*}
is commutative.
\end{df}

The terminology comes from the explicit construction given in the following proposition.

\begin{pr}\label{desrfta}
Let $E$ be a complete (real or complex) locally convex space.
The direct product of all projective tensor products of the form $E\ptn \cdots \ptn E$, including $\mathbb{R}$ (or $\mathbb{C}$), with multiplication extended from the usual tensor algebra $T(E)$ is the formal tensor algebra of~$E$.
\end{pr}
\begin{proof}
It is easy to see that $\Pi\!:=\prod_{k=0}^\infty E^{\ptn k}$ is a unital Arens--Michael algebra with respect to the multiplication  induced by the isomorphisms $E^{\ptn k}\ptn E^{\ptn l}\cong E^{\ptn (k+l)}$; $k,l\in\mathbb{Z}_+$.

Let~$B$ be a Banach algebra and $\psi\!: E \to B$ a continuous linear map such that $\psi(E)^n=0$ for some $n\in\mathbb{N}$. For every $k\in\mathbb{N}$ consider the continuous linear map
$$
\psi_k\!:E^{\ptn k}\to B\!:\, x_1\otimes\cdots\otimes x_k\mapsto \psi(x_1)\cdots\psi(x_k).
$$
Since the  direct product of finitely many locally convex spaces is isomorphic to the sum, we obtain a continuous linear map
$\widehat\psi\!:\prod_{k=0}^{n-1} E^{\ptn k}\to B$, which extends to $ \Pi$ since $\psi(E)^n=0$. It is not hard to see that $\widehat\psi$ is a homomorphism of associative algebras such that $\psi=\widehat\psi \mu$, where $\mu\!:E\to \Pi$ is defined in an obvious way.  Such $\widehat\psi$ is unique since $T(E)$ is dense in $\Pi$. So the conditions of Definition~\ref{formtena} hold and this completes the proof.
\end{proof}

Formal tensor algebras  are useful because the commutant of a Banach algebra of polynomial growth is nilpotent; see Theorem~\ref{RadPG}.

\subsection*{Ordered calculus}

To find an explicit form of ${\mathscr F}_{k}^{\,\mathsf{PG}}$ we use the fact that a  free associative real algebra with $k$ generators is isomorphic to $U(\mathfrak{f}_k)$, where $\mathfrak{f}_k$ is a  free real Lie algebra with $k$ generators. Therefore ${\mathscr F}_{k}^{\,\mathsf{PG}}\cong \widehat U(\mathfrak{f}_k)^{\,\mathsf{PG}}$. Since $\mathfrak{f}_k$ is not finite-dimensional, Theorem~\ref{envhtri} cannot be applied and, as we will see, the envelope of $U(\mathfrak{f}_k)$ is more complicated. However, the proof uses the same scheme as for Theorem~\ref{PGtrifdL}; cf.~\cite{ArOld}.

We need the following well-known fact on the commutant of a free Lie algebra.

\begin{pr}
\label{bascommfree}
Let $k\geqslant 2$ and let $e_1,\ldots,e_k$ denote the generators of $\mathfrak{f}_k$. Then $[\mathfrak{f}_k,\mathfrak{f}_k]$ is an infinitely generated free Lie algebra with the generators
\begin{equation}\label{basisg0}
[e_{j_1},[e_{j_2},[\cdots [e_{j_{s-1}}, e_{j_s}]\cdots]],\qquad j_1\geqslant j_2\geqslant\cdots\geqslant j_{s-1}<j_s,\,s\geqslant 2.
\end{equation}
\end{pr}
For a proof (in the case of a free Lie algebra over an arbitrary commutative ring) see \cite[p.\,64, \S\,2.4.2, Corollary~2.16(ii)]{Ba21} or \cite[p.\,72, Lemma 2.11.16]{BK94}.

Denote by $\Delta$ the set of all pairs $(l,j)$ of positive integers such that $k\geqslant l>j\geqslant 1$. Here and throughout this section, when using the notation $l$ and $j$, it is always assumed that they form a pair $\delta\in\Delta$ and vice versa. It is natural to use polynomials in the operators $\ad e_i$.
For $\delta\in\Delta$ and $\beta=(\beta_1,\ldots, \beta_l)\in\mathbb{Z}_+^l$ put
\begin{equation}\label{notg}
 g_{\delta,\beta}\!:=(\ad{e_1})^{\beta_1}\cdots (\ad{e_{l-1}})^{\beta_{l-1}}(\ad{e_l})^{\beta_l+1}(e_j).
\end{equation}

Reversing the order of indices in~\eqref{basisg0} and reformulating, we get from Proposition~\ref{bascommfree} that
\begin{equation*}
  \{g_{\delta,\beta}\!:\,\delta\in\Delta,\, \beta\in\mathbb{Z}_+^l\}
\end{equation*}
is an algebraic basis of $[\mathfrak{f}_k,\mathfrak{f}_k]$. Denote its linear span by~$V$.

It follows from the PBW theorem that $U(\mathfrak{f}_k)\cong U(\mathfrak{f}_k/[\mathfrak{f}_k,\mathfrak{f}_k])\otimes U([\mathfrak{f}_k,\mathfrak{f}_k])$ as a vector space. Since $[\mathfrak{f}_k,\mathfrak{f}_k]$ is free, we have a linear isomorphism
\begin{equation}\label{Ufdec}
U(\mathfrak{f}_k)\cong \mathbb{R}[e_1,\ldots,e_k ]\otimes T(V),
\end{equation}
where $T(V)$ is the tensor algebra of the vector space $V$. Note that \eqref{Ufdec} also holds for $k=1$ since in this case $\Delta$ is empty and so $V=0$.

Further, every element of $V$ has the form
$$
\sum_{\delta\in\Delta}\Psi_\delta(f_\delta)([e_l,e_j]),
$$
where $f_\delta\in \mathbb{R}[\mu_1,\ldots,\mu_l]$ and
\begin{equation}\label{orcalPsi}
\Psi_\delta\!:\mu_1^{\beta_1}\cdots  \mu_l^{\beta_l} \mapsto (\ad e_1)^{\beta_1}\cdots  (\ad e_l)^{\beta_l}
\end{equation}
is a linear map from $\mathbb{R}[\mu_1,\ldots,\mu_l]$ to the space of linear endomorphisms of $[\mathfrak{f}_k,\mathfrak{f}_k]$. Therefore,
\begin{equation*}
V\cong \bigoplus_{\delta\in\Delta} \mathbb{R}[\mu_1,\ldots,\mu_l].
\end{equation*}

Let $X\!:=\bigsqcup_{\delta\in\Delta} X_\delta$, where $X_{\delta}\!:=\mathbb{R}^l$. When $k=1$, we put $X=\emptyset$.
In general, $X$ is not always a manifold since the components can have different dimensions. However, we use the notation
\begin{equation}\label{cinX}
C^\infty(X)\!:=\bigoplus_{\delta\in\Delta}  C^\infty(X_\delta).
\end{equation}
The natural maps $\mathbb{R}[\mu_1,\ldots,\mu_l]\to C^\infty(X_\delta)$, $\delta\in\Delta$, induce an embedding $V\to C^\infty(X)$.

By analogy with~\eqref{Cinffgdef}, we put
\begin{equation}\label{defCifk}
C^\infty_{\mathfrak{f}_k}\!:=C^\infty(\mathbb{R}^k)\ptn [T](C^\infty(X))
\end{equation}
for $k\in\mathbb{N}$. (Here $[T](C^\infty(X))$ is the formal tensor algebra associated with the Fr\'echet space $C^\infty(X)$; see at the beginning of this section.)

Our aim is to show that $C^\infty_{\mathfrak{f}_k}$ admits a multiplication extending the multiplication on $U(\mathfrak{f}_k)$ and making it a $\mathsf{PGL}$ algebra and, moreover, that $U(\mathfrak{f}_k)\to C^\infty_{\mathfrak{f}_k}$ is an envelope with respect to $\mathsf{PG}$. Here we use the decomposition~\eqref{Ufdec} to define the embedding $U(\mathfrak{f}_k)\to C^\infty_{\mathfrak{f}_k}$ as the tensor product of the linear map $\mathbb{R}[e_1,\ldots,e_k]\to C^\infty (\mathbb{R}^k)$ and the linear map $T(V)\to [T](C^\infty(X))$ induced by $V\to C^\infty(X)$.

The fact that $C^\infty_{\mathfrak{f}_k}$ is an algebra is shown in Theorem~\ref{multCiffk} below. We first prove the existence of an ordered calculus, which is actually multiplicative (Theorem~\ref{fuctcafree}).
The following result is an analogue of Theorem~3.3 in \cite{ArOld}.

\begin{thm}\label{orcalgen}
Let $B$ be in $\mathsf{PGL}$ and  $b_1, \ldots, b_k\in B$. Then the homomorphism  $\theta\!:U(\mathfrak{f}_k)\to B$ determined by $e_j\mapsto b_j$ extends to a linear continuous map $C^\infty_{\mathfrak{f}_k}\to B$.
\end{thm}

For the proof we need the well-known formula
\begin{equation}
\label{adxb}
\exp(is\ad a)(b)=\exp(isa)\,b \exp(-isa),
\end{equation}
where $a$ and $b$ are elements of a Banach algebra and $s$ is a scalar; see, e.g., \cite[Chapter~II, \S\,15, p.\,83, Remark~1]{BS01}.

The following useful lemma is a consequence of~\eqref{adxb}; see details in the proof of Proposition~4.1 in~\cite{ArOld}.
\begin{lm}\label{PGad}
If an element $b$ of a real Banach algebra $B$ is of polynomial growth, then so is the operator $\ad b\!:B\to B$.
\end{lm}

\begin{proof}[Proof of Theorem~\ref{orcalgen}]
It suffices to consider the case when $B$ is a Banach algebra of polynomial growth. Then we can take the
linear maps from $\mathbb{R}[e_1,\ldots,e_k]$ and $T(V)$ to $B$ given by the decomposition in~\eqref{Ufdec}.
We claim that they extends to continuous linear maps from both tensor factors in~\eqref{defCifk}, $C^\infty(\mathbb{R}^k)$ and $[T](C^\infty(X))$.

First, since $b_1, \ldots, b_k$ are of polynomial growth, it follows from Theorem~\ref{fucani} that the ordered polynomial calculus $\mathbb{R}[e_1,\ldots,e_k]\to B$ extends to a continuous ordered calculus $C^\infty(\mathbb{R}^k)\to B$.

Second, we define a continuous linear map $C^\infty(X)\to B$. To do this, for every $\delta\in\Delta$ we extend $\Psi_\delta$ in~\eqref{orcalPsi} to $C^\infty(X_\delta)$. Specifically, put
\begin{equation}\label{Psideext}
\Psi_\delta\!:f\mapsto\frac1{(2\pi)^l}\int_{\mathbb{R}^l}  \widehat f(s_1,\ldots,s_l)\,
 \exp\bigl(i s_1\ad b_1\bigr)\cdots \exp\bigl(i s_l \ad b_l\bigr)([b_l,b_j])\,ds_1\cdots ds_l,
\end{equation}
where the hat means the Fourier transform. By Lemma~\ref{PGad}, the operators $\ad b_1, \ldots, \ad b_k$ are of polynomial growth. Therefore
the function
$$
(s_1,\ldots,s_l)\mapsto\|\widehat f(s_1,\ldots,s_l)\,
 \exp\bigl(i s_1\ad b_1\bigr)\cdots \exp\bigl(i s_l \ad b_l\bigr)([b_l,b_j])\|
 $$
is absolutely integrable. Thus, the integral exists not only in the weak sense (of Gelfand and Pettis), but also in the strong sense (of Bochner) \cite[p.\,80, Theorem 3.7.4]{HP57}. Applying Theorem~\ref{fucani} again, now to the tuple $(\ad b_1, \ldots, \ad b_k)$, we see that $\Psi_\delta$ is indeed  an extension of the ordered calculus in~\eqref{orcalPsi}. From the definition of $C^\infty(X)$ (see \eqref{cinX}), we have an extension of the map $V\to B$ to $C^\infty(X)\to B$.

It follows from~\eqref{adxb} that
$$
 \exp\bigl(i s_l \ad b_l\bigr)([b_l,b_j])=[\exp\bigl(i s_l b_l)\,b_l \exp\bigl(-i s_l b_l)\, , \,\exp\bigl(i s_l b_l)\,b_j\exp\bigl(-i s_l b_l)].
$$
By Theorem~\ref{RadPG}, each Banach algebra of polynomial growth is commutative modulo the radical. Therefore the integrand in~\eqref{Psideext} belongs to  $\Rad B$ for all $s_1,\ldots,s_l$. Since the radical of a Banach algebra is closed and the integral converges in the strong sense, we have that $\Psi_\delta(f)\in\Rad B$ for every $f\in C^\infty(X_\delta)$. Therefore the image of $C^\infty(X)$ is also in $\Rad B$.

By Theorem~\ref{RadPG},  $\Rad B$ is nilpotent. Then we can apply the universal property for $[T](C^\infty(X))$ in Definition~\ref{formtena} and extend the continuous linear map $C^\infty(X)\to B$ constructed above to a continuous linear map from $[T](C^\infty(X) )$ to~$B$. The claim is proved.

It is clear that the composition of the projective tensor product of maps from $C^\infty(\mathbb{R}^k)$ and $[T](C^\infty(X))$ to~$B$ with
the linearization of the multiplication in~$B$,
$$
C^\infty(\mathbb{R}^k)\ptn [T](C^\infty(X))\to B\ptn B\to B,
$$
is the desired extension of $\theta$.
\end{proof}

\subsection*{Extension of multiplication}

The rest of the section concerns the following two  theorems.

\begin{thm}\label{multCiffk}
Let $k\in \mathbb{N}$.

\emph{(A)}~The multiplication on $U(\mathfrak{f}_k)$ extends to a continuous multiplication on $C^\infty_{\mathfrak{f}_k}$.

\emph{(B)}~$C^\infty_{\mathfrak{f}_k}$ is locally in $\mathsf{PG}$.
\end{thm}

\begin{thm}\label{fuctcafree}
Let $k\in \mathbb{N}$. Then the algebra $C^\infty_{\mathfrak{f}_k}$ together with the embedding $U(\mathfrak{f}_k)\to C^\infty_{\mathfrak{f}_k}$ is an envelope with respect to $\mathsf{PG}$, i.e., $\widehat U(\mathfrak{f}_k)^{\,\mathsf{PG}}\cong {\mathscr F}_{k}^{\,\mathsf{PG}} \cong C^\infty_{\mathfrak{f}_k}$.
\end{thm}

These results show that $C^\infty_{\mathfrak{f}_k}$ are deserved to be named algebras of `free $C^\infty$- functions'.
The proof of Theorem~\ref{multCiffk} is quite technical. We give a proof in the case when $k\leqslant 2$, which is a little simpler. The case of general $k$ will be published in a subsequent paper.

On the other hand, Theorem~\ref{fuctcafree} easily follows from
previous results. Taking for granted Theorem~\ref{multCiffk} for every $k\in \mathbb{N}$, Theorem~\ref{fuctcafree} can be deduced is the following way.

\begin{proof}[Proof of Theorem~\ref{fuctcafree}]
Theorem~\ref{multCiffk} implies that  $C^\infty_{\mathfrak{f}_k}$ is in $\mathsf{PGL}$. The desired universal property follows from Theorem~\ref{orcalgen}.
\end{proof}

The main idea of the proof of Theorem~\ref{multCiffk} is the same as for Theorem~4.3 in~\cite{ArOld} and Theorems~\ref{PGqupl} and~\ref{PGSL2R} in this paper. We endow $U(\mathfrak{f}_k)$ with the topology inherited from $C^\infty_{\mathfrak{f}_k}$ and construct a topologically injective homomorphism from $U(\mathfrak{f}_k)$ to a product of algebras of the form $C^\infty(M,{\mathrm T}_p)$, where $M$ is a manifold.

When $k=1$, the set $X$ is empty and then $C^\infty_{\mathfrak{f}_1}=C^\infty(\mathbb{R})$. So  the result is implied by the corresponding assertion in the commutative case, Proposition~\ref{envpol} with $n=1$.

Assume now that $k=2$. In this case $X=\mathbb{R}^2$. For the proof we need several preliminary constructions.

Let $m\in\mathbb{N}$ and $w=(\mathbf{x}_1,\ldots, \mathbf{x}_m)\in X^m$, where  $x_p=(s_p,t_p)$.
For $m\geqslant 1$ consider the following diagonal matrices:
$$
E_w\!:=\diag(0,s_1',\ldots, s_m'),\quad\text{and}\quad
F_w\!:=\diag(t_1',\ldots, t_m',0),
$$
where $t_p'=t_p+\cdots +t_m$ and $s_p'=s_1+\cdots +s_p$. For $m=0$ put $E_\emptyset=F_\emptyset=0$. Take  the Jordan block $J$ of order $m+1$ with zero eigenvalue, i.e.,
\begin{equation*}
J\!:=
\begin{pmatrix}
0& 1&&& \\
&0& 1&& \\
 && \ddots& \ddots&\\
 & &&0 &1\\
 & &&&0
\end{pmatrix}
\end{equation*}
and consider the representations of $\mathfrak{f}_2$ determined by
\begin{equation}\label{zexgen}
\theta_w(e_1)\!:= E_w\quad\text{and}\quad
\theta_w(e_2)=F_w-J.
\end{equation}

We need several simple formulas. First, it is easy to see that if $D\!:=\diag(r_1,\ldots,r_{m+1})$ and
\begin{equation*}
M\!:=
\begin{pmatrix}
0& q_1&0&& \\
&0& q_2&\ddots& \\
 && \ddots& \ddots&0\\
 & &&0 &q_m\\
 & &&&0
\end{pmatrix},
\end{equation*}
then
\begin{equation}\label{DMcomm}
 [D,M]=
\begin{pmatrix}
0& (r_1-r_2)q_1&0&& \\
&0& (r_2-r_3)q_2&\ddots& \\
 && \ddots& \ddots&0\\
 & &&0 &(r_m-r_{m+1})q_m\\
 & &&&0
\end{pmatrix}.
\end{equation}
In particular,
\begin{equation}\label{EwMcomm}
[E_w,M]=
\begin{pmatrix}
0& s_1q_1&0&& \\
&0& s_2q_2&\ddots& \\
 && \ddots& \ddots&0\\
 & &&0 &s_mq_m\\
 & &&&0
\end{pmatrix}
,
\end{equation}

\begin{equation}\label{FwMcomm}
[F_w,M]=
\begin{pmatrix}
0& t_1q_1&0&& \\
&0& t_2q_2&\ddots& \\
 && \ddots& \ddots&0\\
 & &&0 &t_mq_m\\
 & &&&0
\end{pmatrix}
\end{equation}
It is also evident that
\begin{equation}\label{deffc2}
[J,M]\in \mathfrak{c}_2,
\end{equation}
where $\mathfrak{c}_2$ denotes the set of triangular matrices with zeros on the first and second diagonals.

Since $\Delta$ is a one-point set, we write for simplicity $g_\beta$ instead of $g_{\delta,\beta}$. Thus for $\beta=(\beta_1,\beta_2)\in\mathbb{Z}_+^2$ we have from~\eqref{notg} that
\begin{equation}\label{notg2}
 g_{\beta}=(\ad{e_1})^{\beta_1}(\ad{e_2})^{\beta_2+1}(e_1).
\end{equation}

\begin{lm}\label{onefactre}
Let $\beta_1,\beta_2\in \mathbb{Z}$, $m\in\mathbb{N}$  and $w=(\mathbf{x}_1,\ldots, \mathbf{x}_m)\in X^m$, where  $x_p=(s_p,t_p)$ for each $p=1,\ldots,m$.
Then
\begin{equation*}
\theta_w(g_{\beta})-
\begin{pmatrix}
0& s_1^{\beta_1+1}t_1^{\beta_2}&0&& \\
&0& s_2^{\beta_1+1}t_2^{\beta_2}&& \\
 && \ddots& \ddots&0\\
 & &&0 & s_m^{\beta_1+1}t_m^{\beta_2}\\
 & &&&0
\end{pmatrix}
\in \mathfrak{c}_2,
\end{equation*}
\end{lm}
\begin{proof}
It follows from~\eqref{EwMcomm} and~\eqref{notg2}  that
\begin{equation*}
\theta_w(g_0)=[F_w-J,E_w]=
\begin{pmatrix}
0& s_1&0&& \\
&0& s_2&\ddots& \\
 && \ddots& \ddots&0\\
 & &&0 &s_m\\
 & &&&0
\end{pmatrix}.
\end{equation*}

Then, applying inductively~\eqref{deffc2},~\eqref{FwMcomm}  and~\eqref{EwMcomm} and using the fact that $\mathfrak{c}_2$ is a Lie ideal, we get the result.
\end{proof}

For $\boldsymbol{\lambda}=(\lambda_1,\lambda_2)\in\mathbb{R}^2$ and $w\in X^m$ put also
\begin{equation}\label{defwte}
\widetilde\theta_{\boldsymbol{\lambda},w}\!: e_i\mapsto \lambda_i+\theta_w(e_i)\qquad (i=1,2).
\end{equation}
This family of homomorphisms is basic for us. Let
\begin{equation}\label{Qmdef}
Q_m\!:=
\begin{pmatrix}
0& 0&\cdots&&1 \\
&0& 0&& \\
 && \ddots& \ddots&\vdots\\
 & &&0 &0\\
 & &&&0
\end{pmatrix}
\end{equation}
when $m\geqslant 0$ and $Q_0=0$ (a scalar).

\begin{lm}\label{manyQk}
Let $m\in\mathbb{N}$, $w=(\mathbf{x}_1,\ldots, \mathbf{x}_m)\in X^m$ and $\beta^p\in \mathbb{Z}_+^{2}$, where  $x_p=(s_p,t_p)$ and $\beta^p=(\beta^p_1,\beta^p_2)$ for each $p=1,\ldots,m$.
Then
\begin{equation}\label{wtelaw}
\theta_{w}(g_{\beta^1}\cdots g_{\beta^m})= Q_m\,
\prod_{p=1}^m s_p^{\beta^p_1+1}\,t_p^{\beta^p_2}.
\end{equation}
\end{lm}

It is important here that $m$ is both the number of factors on the left-hand side in~\eqref{wtelaw} and the length of the word $w$.

\begin{proof}
If $m=0$, then $w$ is an empty word and so the assertion is trivial.

Let $m\geqslant 1$. It is obvious  that $\theta_{w}(g_{\beta^1}\cdots g_{\beta^m})=\theta_{w}(g_{\beta^1})\cdots\theta_{w}(g_{\beta^m})$.
Note that the product of $m$ strictly upper-diagonal matrices have zero entries everywhere except possibly the upper-right corner. This entry depends only on the entries on the second diagonals and equals the product of the  $p$th entries in the $p$th matrix, where $p$ runs $\{1,\ldots,m\}$. Then it follows from Lemma~\ref{onefactre} that~\eqref{wtelaw} holds.
\end{proof}

Write $U(\mathfrak{f}_2)$ in the form~\eqref{Ufdec}. This means that
every $a\in U(\mathfrak{f}_2)$ can be uniquely written as a finite sum
$$
\sum_{m=0}^\infty \Phi_m(h_m),
$$
where
\begin{multline}\label{Phimdefk2}
\Phi_m\!: \mathbb{R}[\lambda_1,\lambda_2]\otimes  \mathbb{R}[\mu_1,\mu_2]\otimes\cdots\otimes \mathbb{R}[\mu_1,\mu_2]\to U(\mathfrak{f}_2)\!:\\
 f_0\otimes f_1\otimes  \cdots\otimes f_m \mapsto \Phi_0(f_0)\Psi(f_1)([e_2,e_1])\cdots \Psi(f_m)([e_2,e_1]).
\end{multline}
Here $\Phi_0$ is the ordered calculus $\mathbb{R}[\lambda_1,\lambda_2]\to U(\mathfrak{f}_2)\!:\lambda_i\mapsto e_i$ and $\Psi$ is  defined in~\eqref{orcalPsi}.

We consider $h_m$ as an element of $\mathbb{R}[\lambda_1,\lambda_2,s_1, t_1,\ldots,s_m, t_m]$, which we denote, for brevity, by $\mathbb{R}[\boldsymbol{\lambda},\mathbf{x}_1,\ldots,\mathbf{x}_m]$,  i.e., $\boldsymbol{\lambda}=(\lambda_1,\lambda_2)$ and $\mathbf{x}_p=(s_p,t_p)$.

\begin{lm}\label{manyQwek}
Suppose that the hypotheses of Lemma~\ref{manyQk} are satisfied and
$h\in\mathbb{R}[\boldsymbol{\lambda},\mathbf{x}_1,\ldots,\mathbf{x}_m]$. Then
\begin{equation}\label{PhimtrQmma}
\widetilde\theta_{\boldsymbol{\lambda},w}(\Phi_m(h))=s_1\cdots s_m\, h(\lambda_1,\lambda_2+t_1+\cdots+t_m,s_1,t_1,\ldots, s_m,t_m) Q_m
\end{equation}
when $m\geqslant 1$ and  $\widetilde\theta_{\boldsymbol{\lambda},\emptyset}(\Phi_0(h))=h(\lambda_1,\lambda_2)$ when $m=0$.
\end{lm}
\begin{proof}
The case when $m=0$ is obvious. Let $m>0$.

Suppose first that  $h$ is a monomial, i.e.,
$$
h=\lambda_1^{\alpha_1}\lambda_2^{\alpha_2} s_1^{\beta^1_1}t_1^{\beta^1_2}\cdots s_m^{\beta^m_1}t_m^{\beta^m_2}.
$$
It follows from the definition of $\Phi_m$ that
$$
\Phi_m(h)= e_1^{\alpha_1}e_2^{\alpha_2} g_{\beta^1}\cdots g_{\beta^m},
$$
where $\beta^p=(\beta^p_1,\beta^p_2)$.
Then by the definition of $\widetilde\theta_{\boldsymbol{\lambda},w}$ (see~\eqref{defwte}),
$$
\widetilde\theta_{\boldsymbol{\lambda},w}(e_1)=\lambda_1+E_w,\qquad \widetilde\theta_{\boldsymbol{\lambda},w}(e_2)=\lambda_2+F_w-J.
$$
It is easy to see that
$$
\widetilde\theta_{\boldsymbol{\lambda},w}(e_1)Q_m=\lambda_1Q_m
\quad\text{and}\quad
\widetilde\theta_{\boldsymbol{\lambda},w}(e_2)Q_m=(\lambda_2+t_1+\cdots+t_m)Q_m
$$
(because $t_1'=t_1+\cdots +t_m$).
Therefore,
$$
\widetilde\theta_{\boldsymbol{\lambda},w}(e_1^{\alpha_1}e_2^{\alpha_2})Q_m=\lambda_1^{\alpha_1}(\lambda_2+t_1+\cdots+t_m)^{\alpha_2}Q_m.
$$
Note also that $\widetilde\theta_{\boldsymbol{\lambda},w}(a)=\theta_{w}(a) $  for every $a\in[\mathfrak{f}_2,\mathfrak{f}_2]$.
Then by Lemma~\ref{manyQk},
$$
\widetilde\theta_{\boldsymbol{\lambda},w}(e_1^{\alpha_1}e_2^{\alpha_2} g_{\beta^1}\cdots g_{\beta^m})=
\lambda_1^{\alpha_1}(\lambda_2+t_1+\cdots+t_m)^{\alpha_2}
\prod_{p=1}^m s_p^{\beta^p_1+1}\,t_p^{\beta^p_2+1}\, Q_m.
$$
It follows that \eqref{PhimtrQmma} holds for monomials. By the linearity of both sides of the equality, it holds also for every $h\in\mathbb{R}[\boldsymbol{\lambda},\mathbf{x}_1,\ldots,\mathbf{x}_m]$.
\end{proof}

The following simple lemma covers the case when the number of factors is greater than the length of the word. (The reverse case is more complicated but we do not need it.)

\begin{lm}\label{annh}
If $m<m'$, then $\widetilde\theta_{\boldsymbol{\lambda},w}(\Phi_{m'}(h))=0$ for all  $w\in X^m$,
$\boldsymbol{\lambda}\in\mathbb{R}^2$ and
$h\in\mathbb{R}[\boldsymbol{\lambda},\mathbf{x}_1,\ldots,\mathbf{x}_{m'}]$.
\end{lm}
\begin{proof}
It follows from~\eqref{orcalPsi} and~\eqref{Phimdefk2} that $\Phi_{m'}(h)$ is a sum of products of $m'+1$ factors, where all the factors except first is of the form $g_{\beta}$, $\beta\in\mathbb{Z}_+^2$. Lemma~\ref{onefactre} implies that
$\theta_{w}(g_{\beta^1})\cdots\theta_{w}(g_{\beta^{m'}})=0$ when $m<m'$ and this completes the proof.
\end{proof}

\begin{proof}[Proof of Theorem~\ref{multCiffk} in the case when $k=2$]
We prove parts (1) and (2) simultaneously.
It follows from Corollary~\ref{maniTpPGL} that it suffices to construct a homomorphism from $U(\mathfrak{f}_2)$  into a product of algebras of the form $C^\infty(\mathbb{R}^n,{\mathrm T}_{p})$ having a topologically injective linear extension to $C^\infty_{\mathfrak{f}_2}$.
For given $m\in\mathbb{N}$ put
$$
B_m\!:=C^\infty(\mathbb{R}^2\times X^m,{\mathrm T}_{m+1}).
$$
Consider the family $\{\widetilde\theta_{\boldsymbol{\lambda},w}\!:U(\mathfrak{f}_2)\to {\mathrm T}_{m+1}\}$ of homomorphisms defined in~\eqref{defwte}. Varying $\boldsymbol{\lambda}$ and $w$ we have the homomorphism
$$
\pi_m\!:U(\mathfrak{f}_2)\to B_m
$$
determined by
\begin{equation*}
\pi_{m}(a)(\boldsymbol{\lambda},w)\!:= \widetilde\theta_{\boldsymbol{\lambda},w}(a)\qquad (a\in U(\mathfrak{f}_2)).
\end{equation*}

Consider also the homomorphism
$$
\rho\!:U(\mathfrak{f}_2)\to \prod_{m=1}^{\infty} B_m \!:a\mapsto (\pi_m(a)).
$$
Since $\prod B_m$ is in $\mathsf{PGL}$ by Corollary~\ref{maniTpPGL}, it follows from Theorem~\ref{orcalgen} that $\rho$ extends to a continuous linear map from $C^\infty_{\mathfrak{f}_2}$ to $\prod B_m$. To complete the proof it suffices to show that $\rho$ is topologically injective.

Consider $B_m$ as a free $C^\infty(\mathbb{R}^2\times X^m)$-module and denote by $Y_m$ the direct summand that corresponds to the upper-right corner (see the definition of the matrix~$Q_m$ in~\eqref{Qmdef}). It is easy to see that $\prod Y_m$ is a direct factor in $\prod B_m$. So, by Part~(B) of Lemma~\ref{auxtopin}, it suffices to show that the composition of $\rho$ and the projection on $\prod Y_m$ is topologically injective. Denote this composition by $\rho'$.

By definition (see \eqref{defCifk}), $C^\infty_{\mathfrak{f}_2}=C^\infty(\mathbb{R}^2)\ptn [T](C^\infty(X))$, where $X=\mathbb{R}^2$. It follows from Proposition~\ref{desrfta} that  we can write $C^\infty_{\mathfrak{f}_2}=\prod_{n\in\mathbb{Z}_+}C_n$, where $C_n\!:=C^\infty(\mathbb{R}^2\times X^n)$. Then $\rho'$ can be written as $\prod_{n\in\mathbb{Z}_+}C_n \to \prod_{m\in\mathbb{Z}_+}Y_m$. It follows from Lemma~\ref{annh} that $\rho'$ is determined by a lower triangular infinite matrix and so, by Lemma~\ref{auxtopiC}, it suffices to check that $C_m\to Y_m$ is topologically injective for every $m$.

The existence of an extension to an ordered $C^\infty$-calculus given by Theorem~\ref{orcalgen} implies that \eqref{PhimtrQmma} in Lemma~\ref{manyQwek} holds for every $h\in C_m$.
We claim that the endomorphism of $C^\infty(\mathbb{R}^2\times X^m)$ defined by
$$
h\mapsto ((\boldsymbol{\lambda},w)\mapsto s_1\cdots s_m\, h(\lambda_1,\lambda_2+t_1+\cdots+t_m,s_1,t_1,\ldots, s_m,t_m))
$$
is topologically injective. Indeed, it is a composition of a map given by change of variables and a multiplication by a polynomial. The first map is an isomorphism since the change of variables is invertible. The second map is injective since each function taken to $0$ obviously vanishes when
$s_1\cdots s_m\ne 0$ and  also vanishes at other points by continuity.
Moreover,  the range of second map is closed by Proposition~\ref{fgclid} and then it is topologically injective, as well as the composition.
\end{proof}

\end{document}